\documentclass{amsart}

\usepackage[shortlabels]{enumitem}

\usepackage{amsmath, amssymb, amsthm}
\usepackage{color}
\usepackage{url}
\usepackage{tikz}
\usetikzlibrary{shapes,positioning}

\usepackage[a4paper,margin=2.5cm]{geometry}

\newtheorem{theorem}{Theorem}

\newtheorem{proposition}[theorem]{Proposition}
\newtheorem{lemma}[theorem]{Lemma}
\newtheorem{corollary}[theorem]{Corollary}

\theoremstyle{definition}

\DeclareMathOperator{\xc}{xc}
\DeclareMathOperator{\delete}{\backslash}
\DeclareMathOperator{\contract}{\slash}
\DeclareMathOperator{\restrict}{|}
\DeclareMathOperator{\si}{si}

\newcommand{\R}{\mathbb R}

\newcommand{\T}{\mathcal T}

\newcommand{\conv}{\mathrm{conv}}

\newcommand{\rk}{\mbox{rk}}
\newcommand{\stp}{P_{\mathrm{spanning~tree}}}
\newcommand{\arbp}{P_{r\mathrm{-arborescence}}}
\newcommand{\arbdom}{P^\uparrow_{r\mathrm{-arborescence}}}
\newcommand{\cdom}{P^\uparrow_{\mathrm{circuit}}}

\title{Regular Matroids Have Polynomial Extension Complexity}

\author{Manuel Aprile \and Samuel Fiorini}

\begin{document}
\maketitle

\begin{abstract}
We prove that the extension complexity of the independence polytope of every regular matroid on $n$ elements is $O(n^6)$. Past results of Wong~\cite{wong1980integer} and Martin~\cite{Martin91} on extended formulations of the spanning tree polytope of a graph imply a $O(n^2)$ bound for the special case of (co)graphic matroids. However, the case of a general regular matroid was open, despite recent attempts~\cite{kaibel2016extended,weltge2015sizes,GurjarV17}. We also consider the extension complexity of circuit dominants of regular matroids, for which we give a $O(n^2)$ bound.
\end{abstract}


\section{Introduction}

Let $P$ be any polytope in $\R^d$. An \emph{extension} (or \emph{lift}) of $P$ is a polytope $Q \in \R^e$ such that $P=\pi(Q)$ for some affine map $\pi : \R^e \to \R^d$. The \emph{extension complexity} of $P$, denoted by $\xc(P)$, is the minimum number of facets of an extension of $P$. If $Ay \leq b$ is a linear description of $Q$, then $Ay \leq b$, $x = \pi(y)$ is called an \emph{extended formulation} of $P$ since
\(
x \in P \iff \exists y : Ay \leq b,\  x = \pi(y).
\)
Thus the extension complexity of a polytope can also be defined as the minimum number of \emph{inequality constraints} in an extended formulation.

Extended formulations are used and studied for a long time, while extension complexity was formally defined less than ten years ago. This definition was much inspired by the seminal work of Yannakakis~\cite{yannakakis1991expressing}. Recently, researchers tried to pin down the extension complexity of several families of polytopes, mainly in connection with combinatorial optimization. By now, we have a quite good understanding of the extension complexity of the polytopes associated to the main ``textbook paradigms'': flows, matchings, arborescences, traveling salesman tours and stable sets, see~\cite{fiorini2012linear,rothvoss2014matching,goos2018extension}. One notable exception is \emph{matroids}. 


Let $M$ be a matroid. We denote by $E(M)$ the set of elements of $M$ and $\mathcal{I}(M)$ the collection of its independent sets. Also, we denote by $\mathcal{B}(M)$ the collection of its bases. The \emph{independence polytope} of $M$ is the convex hull of the characteristic vectors of independent sets of $M$. Using the notation $P(M)$ for the independence polytope of $M$ and $\chi^I$ for the characteristic vector of independent set $I \in \mathcal{I}(M)$, we have
$$
P(M) = \conv \{\chi^I  \in \{0,1\}^{E(M)} \mid I \in \mathcal{I}(M)\}\,.
$$
Another polytope of interest is the \emph{base polytope} $B(M)$ of matroid $M$. The base polytope is the face of the independence polytope whose vertices are the vectors $\chi^B$, where $B \in \mathcal{B}(M)$. Hence,
$$
B(M) = \{x \in \R^{E(M)} \mid x \in P(M),\ x(E) = \rk(M)\}
$$
where $x(F) := \sum_{e \in F} x_e$ for $F \subseteq E(M)$ and $\rk(M)$ denotes the rank of $M$. Notice that every extended formulation for $P(M)$ yields an extended formulation for $B(M)$ with the same number of inequality constraints, hence $\xc(B(M)) \leq \xc(P(M))$. Letting $n$ denote the number of elements of $M$, we also have $\xc(P(M)) \leq \xc(B(M)) + 2n$ since $P(M) = \{x \in \R^{E(M)} \mid \exists y \in B(M),\ \mathbf{0} \leq x \leq y\}$. 

A \emph{regular matroid} is a matroid that is representable over every field, or, equivalently, that is representable over the reals by a totally unimodular matrix. Regular matroids form a fundamental class of matroids, generalizing graphic and cographic matroids. Let $G$ be a graph. Recall that the elements of the corresponding \emph{graphic matroid} $M(G)$ (also called the \emph{cycle matroid} of $G$) are the edges of $G$, and the independent sets are the edge subsets $F \subseteq E(G)$ that define a forest in $G$. The \emph{cographic matroid} $M^*(G)$ is the dual matroid of $M(G)$. Graphic and cographic matroids are regular. Also, matroids that are both graphic and cographic are exactly those of the form $M(G)$ for some \emph{planar} graph $G$.

Wong~\cite{wong1980integer} and Martin~\cite{Martin91} proved that $\xc(B(M)) = O(|V(G)| \cdot |E(G)|)$ for all graphic matroids $M = M(G)$. It follows directly that $\xc(P(M)) = O(n^2)$ for all graphic or cographic matroids $M$ on $n$ elements. In case $M$ is both graphic and cographic, then $\xc(P(M)) = O(n)$ follows from Williams~\cite{Williams01}.

Let $n$ and $r$ respectively denote the number of elements and rank of $M$. In \cite{kaibel2016extended,weltge2015sizes}, it is claimed that $\xc(P(M)) = O(n^2)$ whenever $M$ is regular. The first version of~\cite{GurjarV17} claimed an even better $O(r \cdot n)$ bound. However, both papers have a fundamental flaw and appear to be difficult to fix\footnote{Actually, \cite{GurjarV17} has been withdrawn after a few months, and \cite{kaibel2016extended} has been recently withdrawn, see \cite{Kaibel2019}.}, and as a result no polynomial bound is currently known. In this paper, we give the first polynomial upper bound on the extension complexity of the independence polytope of a regular matroid.

\begin{theorem}[main theorem]\label{thm:main}
There exists a constant $c_0$ such that $\xc(P(M)) \leq c_0 \cdot n^{6}$ for all regular matroids $M$ on $n$ elements.
\end{theorem}

Our proof of Theorem~\ref{thm:main} is by induction on $n$. We rely on the Seymour's celebrated characterization of regular matroids. (A formal definition of $t$-sum for $t \in [3]$ can be found below, in Section~\ref{sec:decomp}.) 

\begin{theorem}[Seymour's decomposition theorem~\cite{seymour1980decomposition}]\label{thm:seymourdecomposition}
A matroid is regular if and only if it is obtained by means of $1$-, $2$- and $3$-sums, starting from graphic and cographic matroids and copies of a certain $10$-elements matroid $R_{10}$.
\end{theorem}

Let $M$ be a regular matroid on $n$ elements. If $M$ is either graphic, cographic or $R_{10}$, then from~\cite{wong1980integer,Martin91} we directly have $\xc(P(M)) 
\leq c_0 \cdot n^{6}$, provided that $c_0\geq 2$. 
 Next, assume that $M$ is a $t$-sum  of two smaller regular matroids $M_1$ and $M_2$ for some $t \in [2]$ (we write $M=M_1\oplus_t M_2$). Then, using the following simple bound we are done by induction. 
 
\begin{lemma}[see \cite{kaibel2016extended,weltge2015sizes} or \cite{aprile20182}] \label{lem:simplebound} For $t \in [2]$,
$$
\xc(P(M_1 \oplus_t M_2)) \leq \xc(P(M_1)) + \xc(P(M_2))\,.
$$
\end{lemma}

For completeness, here is a proof sketch for Lemma~\ref{lem:simplebound}: if $t = 1$ then $P(M_1 \oplus_t M_2)$ is simply the Cartesian product $P(M_1) \times P(M_2)$, and if $t = 2$ then $P(M_1 \oplus_t M_2)$ can be obtained by intersecting $P(M_1) \times P(M_2)$ with a single hyperplane.

Since we cannot prove Lemma~\ref{lem:simplebound} for $t = 3$, we switch to a different strategy to treat the remaining case. Instead, we prove that $M$ has a special decomposition as described in the next result.

\begin{lemma} \label{lem:star_decomposition}
Let $M$ be a regular matroid on $n$ elements that is neither graphic, nor cographic, nor $R_{10}$, and that is neither a $1$-sum nor a $2$-sum. There exist matroids $M_0$, $M_1$, \ldots, $M_k$, for $1\leq k \leq n/4$, such that:
\begin{enumerate}[label=(\roman*)]
    \item $M_0$ is graphic or cographic and has $|E(M_0)| \leq n$,
    \item $M_1$, \ldots, $M_k$ are mutually disjoint regular matroids, with $|E(M_i)| \leq n/2 + 3$ for $i \in [k]$,
    \item $M$ can be obtained from $M_0$ by simultaneously performing a $3$-sum with $M_i$ for $i \in [k]$.
\end{enumerate}
\end{lemma}

We call a decomposition as in Lemma~\ref{lem:star_decomposition} a \emph{star decomposition} and write $M \stackrel{\star}{=} M_0 \oplus_3 M_1 \oplus_3 \cdots \oplus_3 M_k$. For such (regular) matroids $M$, we prove the following upper bound on the extension complexity of $P(M)$.

\begin{lemma} \label{lem:star_bound}
There exists a constant $c_1$ such that 
$$
\xc(P(M)) \leq c_1 \cdot |E(M_0)|^2 + 16 \sum_{i=1}^k \xc(P(M_i))
$$
for every matroid $M$ that admits a star decomposition $M \stackrel{\star}{=} M_0 \oplus_3 M_1 \oplus_3 \cdots \oplus_3 M_k$.
\end{lemma}

Since the numbers of elements of $M_1$, \ldots, $M_k$ are smaller than the number of elements of $M$ by a constant factor, Lemmas~\ref{lem:star_decomposition} and \ref{lem:star_bound} are enough to prove a polynomial bound on $\xc(P(M))$. Details will be given below in Section~\ref{sec:main_thm}.

Later in this paper, we also consider the \emph{circuit dominant} of a matroid $M$, defined as 
$$
\cdom(M) := \conv\{\chi^C \in \{0,1\}^{E(M)} \mid C \text{ is a circuit of } M\}+\R^{E(M)}_+\,.
$$
We can give a $O(n^2)$-size extended formulation for this polyhedron whenever $M$ is a regular matroid on $n$ elements. Interestingly, the extended formulation can be constructed \emph{globally}, in the sense that it does not need Seymour's decomposition theorem. This is in stark contrast with the case of the independence polytope, for which we do not know how to avoid the decomposition theorem.

Here is an outline of the paper. In Section~\ref{sec:decomp}, we give some background on $t$-sums for $t \in [3]$ and prove Lemma~\ref{lem:star_decomposition}. Then, in Section~\ref{sec:main_thm}, we prove Theorem~\ref{thm:main} assuming Lemma~\ref{lem:star_bound}. The proof of Lemma~\ref{lem:star_bound} occupies the rest of the paper. In Section~\ref{sec:asymmetric} we give a first \emph{asymmetric} extended formulation of $P(M)$ for regular matroids $M$ that are the $3$-sum of two regular matroids $M_1$ and $M_2$, in order to illustrate the main ideas. Unfortunately, this extended formulation is not small enough for our purposes, and we have to use more specifics of the star decomposition $M \stackrel{\star}{=} M_0 \oplus_3 M_1 \oplus_3 \cdots \oplus_3 M_k$, in particular that $M_0$ is graphic or cographic. The graphic case is done in Section~\ref{sec:graphic}, and the cographic case in Section~\ref{sec:cographic}. In Section \ref{sec:cdom}, we provide a small extended formulation for the circuit dominant of any regular matroid. Finally, we discuss some improvements and open problems in Section~\ref{sec:discussion}. Some technical details necessary for the proof of Lemma~\ref{lem:star_bound} can be found in the appendix.

\section{Decompositions} \label{sec:decomp}

The main goal of this section is to prove Lemma~\ref{lem:star_decomposition}. We start by giving a few preliminaries on $t$-sums for $t \in [3]$.

\subsection{$1$-sums, $2$-sums and $3$-sums} In order to define $t$-sums for $t \in \{1,2,3\}$, we restrict to binary matroids. Recall that regular matroids are in particular binary, since they can be represented over every field. Recall also that
a \emph{cycle} of a matroid is the (possibly empty) disjoint union of circuits. Clearly, every matroid is determined by its cycles. (If $M$ is a binary matroid represented by matrix $A \in \mathbb{F}_2^{m \times n}$, then the cycles of $M$ are all solutions $x \in \mathbb{F}_2^n$ of $Ax = \mathbf{0}$.)

Let $M_1$, $M_2$ be binary matroids. Following \cite{seymour1980decomposition}, we define a new binary matroid $M := M_1\Delta M_2$ with $E(M) := E(M_1) \Delta E(M_2)$ such that the cycles of $M_1 \Delta M_2$ are all the subsets of $E(M)$ of the form $C_1 \Delta C_2$, where $C_i$ is a cycle of $M_i$ for $i \in [2]$. We are interested in the following three cases:
\begin{itemize}
    \item $E_1$ and $E_2$ are disjoint, and $E_1, E_2\neq \emptyset$: then we write $M=M_1\oplus_1 M_2$, and say that $M$ is the \emph{1-sum} of $M_1, M_2$;
    \item $E_1$ and $E_2$ share one element $\alpha$, which is not a loop or coloop of $M_1$ or $M_2$, and $|E_1|,|E_2|\geq 3$: then we write $M=M_1\oplus_2 M_2$, and say that $M$ is the \emph{2-sum} of $M_1, M_2$;
    \item $E_1$ and $E_2$ share a $3$-element subset $T=\{\alpha,\beta,\gamma\}$, where $T$ is a circuit of $M_1$ and $M_2$ (called a \emph{triangle}) that does not contain any cocircuit of $M_1$ or  $M_2$, and $|E_1|,|E_2|\geq 7$: then we write $M=M_1\oplus_3 M_2$, and we say that $M$ is the \emph{3-sum} of $M_1, M_2$.
\end{itemize}

In the following, whenever talking about $t$-sums, we implicitly assume that $M_1, M_2$, also called the \emph{parts} of the sum, satisfy the assumptions in the definition of the corresponding operation. A matroid is said to be \emph{connected} (or \emph{2-connected}) if it is not a 1-sum, and \emph{3-connected} if it is not a 2-sum or a 1-sum. A subset $F$ of a matroid $M$ is said to be connected if the restriction $M \restrict F$ is.

\subsection{Star decompositions}

We begin by stating a corollary of~\cite{seymour1980decomposition} that refines the decomposition theorem in the $3$-connected case, and is well-suited to our needs. Its proof can be found in the appendix.
\begin{theorem}\label{thm:seymourdecomposition3conn}
Let $M$ be a $3$-connected regular matroid that is not $R_{10}$. There exists a tree $\mathcal{T}$ such that each node $v \in V(\mathcal{T})$ is labeled with a graphic or cographic matroid $M_v$, each edge $vw \in E(\mathcal{T})$ has a corresponding $3$-sum $M_v \oplus_3 M_w$, and $M$ is the matroid obtained by performing all the $3$-sums operations corresponding to the edges of $\mathcal{T}$ (in arbitrary order). 
\end{theorem}

We will also need the following easy result.

\begin{lemma} \label{lem:weightedtree}
Consider a tree $\mathcal{T}$ with node weights $w : V(\mathcal{T}) \to \R$, and denote by $W$ the total weight of $\mathcal{T}$. Then there is a node $v_0 \in V(\mathcal{T})$ such that each component of $\mathcal{T} - v_0$ has total weight at most $W/2$.
\end{lemma}
\begin{proof} 
Orient each edge $e \in E(\mathcal{T})$ towards the heaviest component of $\mathcal{T} - e$, breaking ties arbitrarily. Now, let $v_0$ be a sink node of this orientation, which exists since $\mathcal{T}$ is a tree. Let $\mathcal{T}_1$, \ldots, $\mathcal{T}_k$ denote the components of $\mathcal{T} - v_0$. Since $v_0$ is a sink, we have $w(\mathcal{T}_i) \leq W - w(\mathcal{T}_i)$ and hence $w(\mathcal{T}_i) \leq W/2$, for all $i \in [k]$.
\end{proof}

\begin{proof}[Proof of Lemma~\ref{lem:star_decomposition}]
Let $\mathcal{T}$ be a decomposition tree for $M$, as described in Theorem~\ref{thm:seymourdecomposition3conn}. Thus each node $v \in V(\mathcal{T})$ is labeled with a graphic or cographic matroid $M_v$. We assign to each node $v$ the weight $w(v) := |E(M) \cap E(M_v)|$, so that the total weight $W$ is $n$. 

Pick a node $v_0$ as in Lemma~\ref{lem:weightedtree}. Let $M_0 := M_{v_0}$ be the (graphic or cographic) matroid corresponding to $v_0$. We have that $M_0$ is a minor of $M$ (see Section \ref{sec:appendixdecomp} of the appendix for definitions and further details) and thus $|E(M_0)| \leq |E(M)|$. Letting $\mathcal{T}_1$, \ldots, $\mathcal{T}_k$ denote the components of $\mathcal{T} - v_0$, define $M_i$ to be the matroid obtained by performing all the $3$-sums corresponding to the edges of $\mathcal{T}_i$. By choice of $v_0$, for $i \in [k]$, we have $|E(M_i)| \leq n/2 + 3$ (the three extra elements are those that get deleted in the $3$-sum $M_0 \oplus_3 M_i$). Finally, we need to argue that $k\leq n/4$: this is implied by the fact that each $M_i$ is part of a 3-sum, hence it has at least 7 elements, at least 4 of which are shared with $M$. Therefore, we have that $M \stackrel{\star}{=} M_0 \oplus_3 M_1 \oplus_3 \cdots \oplus_3 M_k$.
\end{proof}

\section{Proof of main theorem}
\label{sec:main_thm}

In this section, we prove  Theorem~\ref{thm:main} assuming that Lemma~\ref{lem:star_bound} holds. The following technical lemma will be useful.

\begin{lemma} \label{lem:convex}
Let $f : [a,b] \to \R$ be a convex function. For every $\varepsilon \in [0,b-a]$, there holds
$$
f(a+\varepsilon) + f(b-\varepsilon)
\leq f(a) + f(b)\,.
$$
\end{lemma}

We have all ingredients to prove our main theorem.

\begin{proof}[Proof of Theorem~\ref{thm:main}]
Let $M$ be a regular matroid on $n$ elements. We go by induction on $n$. 
If $M$ is either graphic, cographic or $R_{10}$, then $\xc(P(M)) \leq c_0 \cdot n^{6}$, for $c_0\geq 2$. If $M$ is graphic or cographic, this follows from~\cite{wong1980integer,Martin91}. If $M$ is isomorphic to $R_{10}$, we can use the trivial bound $\xc(P(M)) \leq 2^n$. 

Next, assume that $M$ is a $1$- or $2$-sum of regular matroids $M_1$ and $M_2$. If $M$ is a 1-sum, then the bound on $\xc(M)$ follows directly from Lemma~\ref{lem:simplebound}) applying induction. Otherwise, we have $M=M_1\oplus_2 M_2$. For $i \in [2]$, let $n_i := |E(M_i)| 
\geq 3$. We get 
\begin{align*}
\xc(P(M)) &= \xc(P(M_1 \oplus_2 M_2))\\
&\leq \xc(P(M_1)) + \xc(P(M_2)) \quad \text{(by Lemma~\ref{lem:simplebound})}\\
&\leq c_0 \cdot n_1^{6} + c_0 \cdot n_2^{6}
\quad \text{(by induction)}\\
&\leq c_0 \cdot (\underbrace{n_1+n_2-3}_{=n-1})^{6} + 
c_0 \cdot 3^{6}
\quad \text{(by Lemma~\ref{lem:convex})}\\
&\leq c_0 \cdot n^{6} \quad \text{(since $n \geq 4$)}\,.
\end{align*}

In the remaining case, $M$ is neither graphic, nor cographic and not a $1$- or $2$-sum. By Lemma~\ref{lem:star_decomposition}, $M$ has a star decomposition $M \stackrel{\star}{=} M_0 \oplus_3 M_1 \oplus_3 \cdots \oplus_3 M_k$. For $i \in \{0\} \cup [k]$, let $n_i := |E(M_i)|$. Notice that $3k \leq n_0 \leq n$, $7 \leq n_i \leq n/2 + 3$ for $i \in [k]$
 and $\sum_{i=0}^k n_i = n + 6k$, thus $\sum_{i=1}^k n_i \leq n + 3k$. This time, we bound $\xc(P(M))$ as follows:
\begin{align*}
\xc(P(M)) &= \xc(P(M_0 \oplus_3 M_1 \oplus_3 \cdots \oplus_3 M_k))\\
&\leq c_1 \cdot |E(M_0)|^2 + 16 \sum_{i=1}^k \xc(P(M_i)) \quad \text{(by Lemma~\ref{lem:star_bound})}\\
&\leq c_0 \cdot n_0^2 + 16c_0 \cdot \sum_{i=1}^k n_i^{6}
\quad \text{(by induction, provided that $c_0 \geq c_1$)}\\
&\leq c_0 \cdot \underbrace{n_0^2}_{\leq n^2} + 16c_0 \cdot \big( 3 \cdot (n/2+3)^{6}+\underbrace{k}_{\leq n/4} \cdot 7^{6} \big) \quad \text{(by Lemma~\ref{lem:convex})\footnotemark}\\
&\leq c_0 \cdot n^{6} \quad \text{(if $n$ is large enough, in particular $n \geq 123$)}\,.
\end{align*}
\footnotetext{By applying Lemma~\ref{lem:convex} repeatedly, and by reordering, we can assume that $n_1,\dots, n_h$ for some $h$ are equal to $n/2+3$ and $n_{h+2,\dots},\dots, n_k$ are equal to 7, with $n_{h+1}$ possibly in between. Since $k\leq n/4$, a simple calculation implies $h+1\leq 3$, and the bound follows. 
}
If $n$ is too small for the last inequality to hold, we use the direct bound $\xc(P(M)) \leq 2^n \leq c_0 \cdot n^{6}$, which holds provided that $c_0$ is large enough. 
\end{proof}

\section{Asymmetric formulations for $3$-sums} \label{sec:asymmetric}

In this section we take one big conceptual step towards a proof of Lemma~\ref{lem:star_bound}. Using the characterization of bases in a $3$-sum, it is easy to obtain an extended formulation for $P(M_1 \oplus_3 M_2)$ whose size is bounded by $c_2 \cdot \xc(P(M_1)) + c_2 \cdot \xc(P(M_2))$ for some constant $c_2 \geq 1$. We call this type of formulation \emph{symmetric}\footnote{We point out that the symmetry in extended formulation was studied before, with a different meaning, see e.g.~\cite{yannakakis1991expressing,KaibelPT12}. In contrast, the adjective ``symmetric'' is used here in an illustrative way and does not have a mathematically precise meaning.}, since $M_1$ and $M_2$ play similar roles. Unless $c_2 = 1$, symmetric formulations do not lead to a polynomial size extended formulation for $P(M)$ for all regular matroids $M$. Since the best constant we know of is $c_2 = 4$, we do not see how to prove Theorem~\ref{thm:main} in this way. 

Instead, we propose an \emph{asymmetric} formulations for $P(M_1 \oplus_3 M_2)$, that is, an extended formulation of size at most $c_3 \cdot \xc(P(M_1)) + c_4 \cdot \xc(P(M_2))$ where $1 \leq c_3 \leq c_4$ and $c_3$ is as small as possible, at the cost of making $c_4$ large. This is our first insight.

Our intuition for asymmetric formulations mainly comes from optimization. Let $M_1$ and $M_2$ be binary matroids sharing a triangle $T := E(M_1) \cap E(M_2)$. In order to find a maximum weight independent set in $M_1 \oplus_3 M_2$ we first solve \emph{several} subproblems in $M_2$, then use this to define weights for the elements of the triangle $T := E(M_1) \cap E(M_2)$ and then solve a \emph{single} optimization problem over $M_1$, where the elements of $E_1\setminus T$ keep their original weights. Eventually, this leads to an asymmetric formulation with $c_3 = 2$ and $c_4 = 16$. (Roughly speaking, the reason why this gives $c_3 = 2$ and not $c_3 = 1$ is that in order to convert the optimization algorithm into an extended formulation, we need to distinguish between two types of objective functions. Actually, our point of view below will be slightly different.)

Next, we quickly explain how $c_3$ can be lowered to $1$ when the term $\xc(P(M_1))$ is replaced by the extension complexity of a certain \emph{pair} of polytopes depending on $M_1$ and $T$. This is our second insight, and will serve as a conceptual basis for our proof of Lemma~\ref{lem:star_bound}. 

Finally, we discuss how things change when, instead of being defined by a single $3$-sum, $M$ is defined by a star decomposition. Hence, instead of having a single triangle $T$, we will have $k \geq 1$ disjoint triangles $T_1$, \ldots, $T_k$.

\subsection{Preliminaries}

We state some facts on $3$-sums that will be useful below. 
If $M$ is a matroid and $e \in E(M)$, we denote by $M \delete e$ the matroid obtained from $M$ by deleting $e$ and by $M \contract e$ the matroid obtained from $M$ by contracting $e$. These notations carry on to subsets $F \subseteq E(M)$. Also, recall that $M \restrict F$ denotes the restriction of $M$ to $F$.
For the rest of the section, we consider a binary matroid $M$ such that $M = M_1\oplus_3 M_2$, where $M_1$ and $M_2$ are binary matroids. Let $T := E(M_1) \cap E(M_2)$ be the triangle on which $M_1$ and $M_2$ are attached to form their $3$-sum. Our first lemma lists some useful well known facts. We refer to \cite{oxley2006matroid} and \cite{schrijver2003combinatorial} for proofs.

\begin{lemma}\label{lem:3sumfacts}
If $M=M_1\oplus_3 M_2$, then the following hold.
  \begin{enumerate}[label=(\roman*)]
      \item $\rk(M)=\rk(M_1)+\rk(M_2)-2$.
      \item The flats of $M$ are of the form $F_1\Delta F_2$, where $F_i$ is a flat of $M_i$ for $i \in [2]$, with $F_1\cap T=F_2\cap T$.
      \item The circuits of $M$ are of the form $C_1\Delta C_2$, where $C_i$ is a circuit of $M_i$ for $i \in [2]$, with $C_1\cap T=C_2\cap T$.
      \item Let $F \subseteq E(M)$ such that $F \subseteq E(M_1)$ (resp.\ $F \subseteq E(M_2)$). Then $M \restrict F = M_1 \restrict F$ (resp.\ $M \restrict F = M_2 \restrict F$). In particular, $I \subseteq F$ is an independent set of $M$ if and only if it is an independent set of $M_1$ (resp.\ $M_2$).
  \end{enumerate}
\end{lemma}

Our next lemma gives a characterization of the bases of a $3$-sum. Its proof can be found in the appendix. 

\begin{lemma}\label{lem:3sumbases}
  Let $M=M_1\oplus_3 M_2$ and let $T := E(M_1) \cap E(M_2)$. A subset $B \subseteq E(M)$ is a basis of $M$ if and only if one of the following holds for some $i \in [2]$ 
  \begin{enumerate}[label=(\roman*)]
    \item $B = B_i \cup (B_{3-i}-t_1-t_2)$, where $B_i$ is a basis of $M_i$ disjoint from $T$ and $B_{3-i}$ is a basis of $M_{3-i}$ containing two elements $t_1, t_2 \in T$.
    \item $B = (B_i - t_1) \cup (B_{3-i} - t_2)$, where $B_i$ is a basis of $M_i$ intersecting $T$ in a single element $t_1$, $B_{3-i}$ is a basis of $M_{3-i}$ intersecting $T$ in a single element $t_2$ distinct from $t_1$, and moreover $B_i - t_1 + t_3$ is a basis of $M_i$ and $B_{3-i} - t_2 + t_3$ is a basis of $M_{3-i}$ where $t_3$ denotes the third element of $T$.
  \end{enumerate}
\end{lemma}

We conclude these preliminaries with properties of  connected flats in a $3$-sum for later use. Our interest for these flats is motivated by the well known fact that for any (loopless) matroid $M$, 
$$
P(M)=\{x \in \R^{E(M)}_+ \mid \forall \text{ connected flat } F\subseteq E(M) : x(F) \leq \rk(F)\}\,.
$$
See, e.g., \cite{schrijver2003combinatorial}. We refer the reader to the appendix for the proof of Lemma~\ref{lem:3sumflats}.

\begin{lemma}\label{lem:3sumflats}
Let $M=M_1\oplus_3 M_2$ and let $T := E(M_1) \cap E(M_2)$. If $F$ is a connected flat of $M$, then $F$ satisfies one of the following.
 
\begin{enumerate}[label=(\roman*)]
\item $F \subseteq E(M_i)$ for some $i \in [2]$ and $F$ is a connected flat of $M_i$.
\item There are connected flats $F_1$, $F_2$ of $M_1$, $M_2$ respectively such that $F = F_1 \Delta F_2$, $F_1 \cap T = F_2\cap T$ is a singleton, and $\rk(F) = \rk (F_1)+\rk(F_2)-1$.
\item There are connected flats $F_1$, $F_2$ of $M_1$, $M_2$ respectively such that $F = F_1 \Delta F_2$, $F_1 \cap T = F_2\cap T$ is the whole triangle $T$, and $\rk(F) = \rk(F_1)+\rk(F_2)-2$.
\end{enumerate}
\end{lemma}

For simplicity (and without loss of generality) throughout the paper we assume that the matroids considered have no loop.

\subsection{A first asymmetric formulation}

We start by recalling a well known result by Balas \cite{balas1979disjunctive}, that we will need below. We state a refinement of it that is proved in \cite{weltge2015sizes}.
\begin{proposition}\label{prop:Balas}
Let $P_1,\dots,P_k\in\R^n$ be polytopes of dimension at least 1. Then $$\xc(\conv \bigcup_{i=1}^k P_i)\leq \sum_{i=1}^k \xc(P_i).$$ 
\end{proposition}

Let $M_1$, $M_2$ be binary matroids sharing a triangle $T := \{\alpha,\beta,\gamma\}$, and let $M := M_1 \oplus_3 M_2$. We give an extended formulation for $P(M)$ showing that
$$
\xc(P(M_1\oplus_3 M_2))\leq 2\xc(P(M_1))+16\xc(P(M_2))\,.
$$

For $X \subseteq T$, we consider the convex hull $P(M_2 \delete T,X)$ of all characteristic vectors $\chi^I \in \{0,1\}^{E(M_2 \delete T)}$ where $I \subseteq E(M_2 \delete T)$ is an independent set of $M_2$ whose span $F$ satisfies $F \cap T \subseteq X$. Observe that
\begin{align*}
P(M_2 \delete T,\emptyset) &= 
P(M_2 \contract \alpha \contract \beta \contract \gamma) = P(M_2 \contract T)\,,\\ 
P(M_2 \delete T,T) &= 
P(M_2 \delete \alpha \delete \beta \delete \gamma) = P(M_2 \delete T)\,,\\ 
P(M_2 \delete T,\{\alpha\}) &= P(M_2 \contract \beta \delete \alpha \delete \gamma)\cap P(M_2 \contract \gamma \delete \alpha \delete \beta)
\end{align*}
and similarly for $P(M_2 \delete T,\{\beta\})$ and $P(M_2 \delete T,\{\gamma\})$ (the last equality follows from matroid intersection).

\begin{proposition}\label{prop:asymmetric}
Let $M_1$, $M_2$ be binary matroids sharing a triangle $T := \{\alpha,\beta,\gamma\}$, and let $M := M_1 \oplus_3 M_2$. Define $P'(M_2)$ as 
$$
P'(M_2) := \conv \left( 
P(M_2 \delete T,\emptyset) \times \{\mathbf{0}\}\ \cup\ \bigcup_{t \in T} P(M_2 \delete T,\{t\}) \times \{\mathbf{e}_{t}\}
\ \cup\ P(M_2 \delete T,T) \times \{\mathbf{e}_\alpha + \mathbf{e}_\beta\}\right)
$$
and $P''(M_2)$ similarly, replacing the last polytope in the union by $P(M_2 \delete T,T) \times \{\mathbf{e}_\beta + \mathbf{e}_\gamma\}$. If we let
\begin{align*} \label{eq:EF}
Q(M) := \Big\{(x^1,x^2)\in \R^{E(M)} \mid \exists \ x^{T'}, x^{T''}\in \R^{T} :\:& 
(x^1,x^{T'})\in P(M_1),\ (x^1,x^{T''})\in P(M_1)\\ & (x^2,x^{T'})\in P'(M_2),\ (x^2,x^{T''})\in P''(M_2)\Big\}\,
\end{align*}
then $P(M)=Q(M)$. In particular, we have $\xc(P(M_1\oplus_3 M_2))\leq 2\xc(P(M_1))+16\xc(P(M_2))$.
\end{proposition}

\begin{proof}
To prove $P(M) \subseteq Q(M)$, we show that $B(M) \subseteq Q(M)$ using Lemma~\ref{lem:3sumbases}, and observe that $Q(M)$ is of antiblocking type (this follows from the fact that $P(N)$ is of antiblocking type for every matroid $N$). Let $B \in \mathcal{B}(M)$ be a basis of $M$. We distinguish cases as in Lemma~\ref{lem:3sumbases}. For the sake of conciseness, we skip the cases that follow from other cases by symmetry, and omit the conditions on the bases $B_1$ and $B_2$ (these can be found in the statement of the lemma).\medskip

\noindent (i) First, assume $B = B_1 \cup (B_2 - \alpha - \beta)$ and let $I_2 := B_2 - \alpha - \beta$. The span of $I_2$ (in $M_2$) is disjoint from $T$. Hence, we have $\chi^{I_2} \in P(M_2 \delete T, \emptyset)$ and $(\chi^{I_2},\mathbf{0}) \in P'(M_2) \cap P''(M_2)$. Then it is easy to check that $\chi^B \in Q(M)$ by setting $x^{T'} = x^{T''} := \mathbf{0}$. 

Next, assume $B = (B_1 - \alpha - \beta) \cup B_2$. Then it is easy to check that $\chi^B \in Q(M)$ by setting $x^{T'} := \mathbf{e}_\alpha + \mathbf{e}_\beta$ and $x^{T''} := \mathbf{e}_\beta + \mathbf{e}_\gamma$.\medskip

\noindent (ii) $B = (B_1 - \alpha) \cup (B_2 - \beta)$. Then we see that $\chi^B \in Q(M)$ by setting $x^{T'} = x^{T''} := \mathbf{e}_\alpha$.\medskip

To prove $Q(M) \subseteq P(M)$, let $F \subseteq E(M)$ be any connected flat and let $x=(x^1,x^2)$ be any point of $Q(M)$. We have to show that $x^1(F\cap E(M_1))+x^2(F\cap E(M_2)) \leq \rk(F)$. We use Lemma~\ref{lem:3sumflats}. If $F \subseteq E(M_1)$, or $F \subseteq E(M_2)$, there is nothing to show. Hence we may focus on cases (ii) and (iii) of the lemma. Therefore, $F = F_1 \Delta F_2$, where $F_i$ is a connected flat of $M_i$ for $i \in [2]$.\medskip

\noindent (ii) $F_1 \cap T = F_2 \cap T$ is a singleton, and $\rk(F)=\rk(F_1)+\rk(F_2)-1$. First, assume that $F_1 \cap T = F_2 \cap T = \{\alpha\}$. Let $x^{T'}$ be such that $(x^1,x^{T'})\in P(M_1)$ and $(x^2,x^{T'}) \in P'(M_2)$. Clearly, we have $x^1(F\cap E(M_1))+x^{T'}_{\alpha} \leq \rk(F_1)$. We claim that
\begin{equation} \label{eq:2ndpart}
x^2(F \cap E(M_2)) \leq \rk(F_2)-1+x^{T'}_{\alpha}\,.
\end{equation}
This concludes the proof for this case as summing the two inequalities, we get the desired inequality. To prove the claim, we may assume that $(x^2,x^{T'})$ is a vertex of $P'(M_2)$. We consider all the possible subcases one after the other.
\begin{itemize}
    \item If $(x^2,x^{T'}) \in P(M_2 \delete T,\emptyset) \times \{\mathbf{0}\}$, then since the rank of $F \cap E(M_2)$ in $M_2 \contract T$ is $\rk(F_2)-1$, \eqref{eq:2ndpart}~holds.
    
    \item If $(x^2,x^{T'}) \in P(M_2 \delete T, \{\alpha\}) \times \{\mathbf{e}_{\alpha}\}$, then $x^{T'}_{\alpha}=1$ and \eqref{eq:2ndpart}~holds.
    
    \item If $(x^2,x^{T'}) \in P(M_2 \delete T, \{\beta\}) \times \{\mathbf{e}_{\beta}\}$, then in particular $x^2\in P(M_2 \contract \alpha \delete \beta \delete \gamma)$, and the rank of $F \cap E(M_2)$ in $M_2 \contract \alpha \delete \beta \delete \gamma$ is $\rk(F_2)-1$. Hence, \eqref{eq:2ndpart}~holds. 
    
    \item If $(x^2,x^{T'}) \in P(M_2 \delete T, \{\gamma\}) \times \{\mathbf{e}_{\gamma}\}$ then a similar argument as in the previous case applies.
    
    \item If $(x^2,x^{T'}) \in P(M_2 \delete T, T) \times \{\mathbf{e}_{\alpha} + \mathbf{e}_{\beta}\}$, then $x^{T'}_{\alpha}=1$ and \eqref{eq:2ndpart}~holds.
\end{itemize}

The above argument can be easily adapted in case $F_1 \cap F_2 = \{\beta\}$. If $F_1 \cap  F_2 = \{\gamma\}$, one needs to use the variables $x^{T''}$ instead. We can show similarly as above that, whenever 
$(x^1,x^{T''})\in P(M_1)$ and $(x^2,x^{T''}) \in P''(M_2)$,
$$
x^2(F \cap E(M_2)) \leq \rk(F_2)-1+x^{T''}_{\gamma}\,.
$$
Together with $x^1 (F \cap E(M_1))+x^{T''}_{\gamma} \leq \rk(F_1)$, this concludes this case.\medskip
    
\noindent (iii) $F_1 \cap T = F_2 \cap T = T$ and $\rk(F)=\rk(F_1)+\rk(F_2)-2$. Again, let $x^{T'}$ be such that $(x^1,x^{T'})\in P(M_1)$ and $(x^2,x^{T'}) \in P'(M_2)$.
We have $x^1(F\cap E(M_1))+x^{T'}_{\alpha}+x^{T'}_{\beta}+x^{T'}_{\gamma}\leq \rk(F_1)$. We claim that 
\begin{equation} \label{eq:2ndpartbis}
x^2(F\cap E(M_2)) \leq \rk(F_2)-2+x^{T'}_{\alpha}+x^{T'}_{\beta}+x^{T'}_{\gamma}
\end{equation}
holds, which concludes the proof for this case as summing the two inequalities we get the desired inequality. As above, we consider all subcases in order to establish~\eqref{eq:2ndpartbis}.

\begin{itemize}
    \item If $(x^2,x^{T'}) \in P(M_2 \delete T,\emptyset) \times \{\mathbf{0}\}$, then since the rank of $F \cap E(M_2)$ in $M_2 \contract T$ is $\rk(F_2)-2$, \eqref{eq:2ndpartbis}~holds.
    
    \item If $(x^2,x^{T'}) \in P(M_2 \delete T, \{t\}) \times \{\mathbf{e}_t\}$ for some $t \in T$ then $x^{T'}_{t}=1$. Since the rank of $F \cap E(M_2)$ in the corresponding minor of $M_2$ is $\rk(F_2)-1$, \eqref{eq:2ndpartbis} holds. 

    \item If $(x^2,x^{T'}) \in P(M_2 \delete T, T) \times \{\mathbf{e}_{\alpha} + \mathbf{e}_{\beta}\}$, then $x^{T'}_{\alpha}=x^{T'}_{\beta}=1$ and \eqref{eq:2ndpartbis}~holds.
\end{itemize}
\end{proof}

\subsection{Making the formulation smaller}

In the upper bound on $\xc(P(M_1 \oplus_3 M_2))$ from Proposition~\ref{prop:asymmetric}, the term $2 \xc(P(M_1))$ comes from the constraints $(x^1,x^{T'}) \in P(M_1)$, $(x^1,x^{T''}) \in P(M_1)$ that are part of the extended formulation. In order to make this term smaller, and hence the formulation more compact on the $M_1$ side, it suffices to find a smaller extended formulation for the polytope 
$$
Q_T(M_1) := \{(x^1,x^{T'},x^{T''}) \mid  (x^1,x^{T'}) \in P(M_1), (x^1,x^{T''}) \in P(M_1)\}\,.
$$

Now with a bit more thinking, we see that it is not necessary to express $Q_T(M_1)$ exactly. In fact, the proof goes through as long as our extended formulation for that part is contained in $Q_T(M_1)$ and contains 
 $$
P_T(M_1) := \conv\{ (\chi^I,\chi^{I'},\chi^{I''}) \mid I \cup I',I \cup I' \in \mathcal{I}(M_1),\ I'=I'' \mbox{ or } I'=\{\alpha,\beta\},\ I''=\{\beta,\gamma\}\}\,.
$$ 
In other words, all we need is an extended formulation \emph{for the pair} of nested polytopes $(P_T(M_1),Q_T(M_1))$.

Before stating our next result, we give some terminology relative to pairs of polytopes. If $P \subseteq Q \subseteq \R^d$ are nested polytopes, an \emph{extension} of the pair $(P,Q)$ is an extension of some polytope $R$ such that $P \subseteq R \subseteq Q$. Similarly, an \emph{extended formulation} for $(P,Q)$ is an extended formulation for such a polytope~$R$. The \emph{extension complexity} of $(P,Q)$ is defined as $\xc(P,Q) := \min \{\xc(R) \mid R$ polytope, $P \subseteq R \subseteq Q\}$. 

The proof of the following is simple and omitted.

\begin{proposition} \label{prop:asymmetricpair}
Let $M_1$, $M_2$ be binary matroids sharing a triangle $T$, and let $M := M_1 \oplus_3 M_2$. Let $P'(M_2), P''(M_2)$ be defined as in Proposition~\ref{prop:asymmetric}, and $P_T(M_1), Q_T(M_1)$ as above. If $R_T(M_1)$ is any polytope such that $P_T(M_1)\subseteq R_T(M_1)\subseteq Q_T(M_1)$, then
\begin{align*} 
P(M)=\{(x^1,x^2)\in \R^{E(M)} \mid \exists \ x^{T'}, x^{T''} \in \R^{T} :\: &(x^1,x^{T'},x^{T''})\in R_T(M_1)\\ & (x^2,x^{T'})\in P'(M_2), (x^2,x^{T''})\in P''(M_2)\}.
\end{align*}
In particular, we have $\xc(P(M_1\oplus_3 M_2))\leq \xc(P_T(M_1),Q_T(M_1))+16\xc(P(M_2))$.
\end{proposition}

\subsection{Dealing with several $3$-sums simultaneously}

We would now like to further extend the above results to the setting where $M = M_0 \oplus_3 \dots \oplus_3 M_k$ for some binary matroids $M_0$, $M_1$, \ldots, $M_k$ such that each $M_i$, $i \in [k]$ shares a triangle $T_i := \{\alpha_i,\beta_i,\gamma_i\}$ with $M_0$ and is disjoint from $M_j$ for all $j \in [k]$ such that $j \neq i$. Notice that a true star decomposition satisfies more conditions (see (i) and (ii) in Lemma~\ref{lem:star_decomposition}). In particular, $M_0$ is required to be graphic or cographic. This will be exploited in the next section. Here, $M_0$ can be any binary matroid. 

For simplicity, we partition $E(M_0)$ into $T_1, \dots, T_k$ and $E_0 := E(M_0) \setminus (T_1 \cup \cdots \cup T_k)$. We let

\begin{align*}
P_{T_1,\dots,T_k}(M_0) :=&\conv\Big\{(\chi^{J_0},\chi^{J_1'}, \chi^{J_1''},\dots, \chi^{J_k'}, \chi^{J_k''}) \in
\R^{E_0} \times \R^{T_1} \times \R^{T_1} \times \cdots \times \R^{T_k} \times \R^{T_k} \mid\\
&\qquad \qquad \forall J_1^*, \ldots, J_k^* : \forall i \in [k] : J_i^* \in \{J_i',J_i''\} : J_0 \cup J_1^* \cup \dots \cup J_k^* \in \mathcal{I}(M_0), \text{ and}\\
&\qquad \qquad \forall i \in [k] : J_i'=J_i'' \text{ or } J_i'=\{\alpha_i,\beta_i\},\ J_i''=\{\beta_i,\gamma_i\}\Big\}\,,\\
Q_{T_1,\dots,T_k}(M_0) :=&\Big\{ (x^0,x^{T_1'},x^{T_1''},\ldots,x^{T_k'},x^{T_k''}) \in \R^{E_0} \times \R^{T_1} \times \R^{T_1} \times \cdots \times \R^{T_k} \times \R^{T_k} \mid \\
&\qquad \forall T_1^*, \ldots, T_k^* : \forall i \in [k] : T_i^* \in \{T_i',T_i''\} : (x^0,x^{T_1^*},\dots,x^{T_k^*}) \in P(M_0)\Big\}.
\end{align*}

\begin{proposition}\label{prop:asymmetricstar}
Let $M = M_0 \oplus_3 M_1 \oplus_3 \cdots \oplus_3 M_k$ where $M_0, M_1, \ldots, M_k$ are binary matroids such that $M_1$, \ldots, $M_k$ are mutually disjoint. For $i \in [k]$, define $P'(M_i), P'(M_i)$ as in Proposition~\ref{prop:asymmetric}. If  $R_{T_1,\dots,T_k}(M_0)$ is any polytope such that $P_{T_1,\dots,T_k}(M_0)\subseteq R_{T_1,\dots,T_k}(M_0)\subseteq Q_{T_1,\dots,T_k}(M_0)$, then
\begin{align} \label{eq:EFstar}
P(M)=\Big\{(x^0,x^1,\dots,x^k)\in \R^{E(M)} \mid&\: \exists \ x^{T_1'}, x^{T_1''} \in \R^{T_1}, \ldots, x^{T_k'},  x^{T_k''} \in \R^{T_k} : &\\
\nonumber&\quad (x^0,x^{T_1'},x^{T_1''}, \dots, x^{T_k'},x^{T_k''})\in R_{T_1,\dots,T_k}(M_0)&\\ 
\nonumber&\quad \forall i \in [k] : (x^i,x^{T_i'})\in P'(M_i),\ (x^i,x^{T_i''})\in P''(M_i) \Big\}.
\end{align}
In particular, we have $\xc(P(M)) \leq \xc(P_{T_1,\dots,T_k}(M_0),Q_{T_1,\dots,T_k}(M_0))+16\sum_{i=1}^k\xc(P(M_i))$.
\end{proposition}
\begin{proof}
We will proceed by induction on $k$. Notice that the base case $k=1$ is Proposition~\ref{prop:asymmetricpair}. Let $k>1$, and let $M' := M_0\oplus_3\dots \oplus_3 M_{k-1}$, so that $M = M'\oplus_3 M_k$, with $T_k$ being the common triangle. Denote by $Q(M)$ the polytope in the right-hand side of~\eqref{eq:EFstar}. By induction, we have that $P(M')=Q(M')$, which we will use below. Let 
\begin{align*}
    R(M') := \Big\{ (x^0,x^1,\dots,x^{k-1}, x^{T_k'},x^{T_k''}) \mid & \ \exists \ x^{T_1'},x^{T_1''} \in \R^{T_1}, \dots, x^{T_{k-1}'}, x^{T_{k-1}''} \in \R^{T_{k-1}} : &\\
    & \quad (x^0,x^{T_1'},x^{T_1''}, \dots, x^{T_k'},x^{T_k''})\in R_{T_1,\dots,T_k}(M_0)&\\ 
    & \quad \forall \ i \in [k] : (x^i,x^{T_i'})\in P'(M_i),\ (x^i,x^{T_i''})\in P''(M_i)\Big\}.
\end{align*} 
We claim that $P_{T_k}(M')\subseteq R(M')\subseteq Q_{T_k}(M')$. Then, by Proposition~\ref{prop:asymmetricpair}, we have that 
\begin{align*}
P(M)=\{(x^0,x^1,\dots,x^k)\in \R^{E(M)} \mid \ \exists \ x^{T_k'}, x^{T_k''} \in \R^{T_k} :\:& (x^0,x^1,\dots,x^{k-1}, x^{T_k'},x^{T_k''})\in R(M') \\
& (x^k,x^{T_k'})\in P'(M_k),\ (x^k,x^{T_k''}) \in P''(M_k) \}.
\end{align*}
But the latter, by definition of $R(M')$, is exactly $Q(M)$, which concludes the proof. We prove the claim below. 

To show $P_{T_k}(M')\subseteq R(M')$, one proceeds as in the proof of Proposition~\ref{prop:asymmetric}, by showing that for every vertex of $P_{T_k}(M')$ there are choices for $x^{T_1'}, x^{T_1''}, \ldots, x^{T_{k-1}'},x^{T_{k-1}''}$ that satisfy all the constraints in $R(M')$. 

To show $R(M')\subseteq Q_{T_k}(M')$, it suffices to prove that whenever $(x^0,x^1,\dots,x^{k-1},x^{T_k'},x^{T_k''}) \in R(M')$, we have $(x^0,x^1,\dots,x^{k-1}, x^{T_k'}) \in P(M')$ and $(x^0,x^1,\dots,x^{k-1}, x^{T_k''}) \in P(M')$. Since $R_{T_1,\ldots,T_k}(M_0) \subseteq Q_{T_1,\ldots,T_k}(M_0)$, we have $(x^0,x^1,\dots,x^{k-1}, x^{T_k'}) \in Q(M')$ and $(x^0,x^1,\dots,x^{k-1},x^{T_k''}) \in Q(M')$. Using $P(M') = Q(M')$, this observation concludes the proof.
\end{proof}

\section{Smaller formulation for star decompositions: the graphic case} \label{sec:graphic}

In this section we first review Wong's extended formulation for the spanning tree polytope~\cite{wong1980integer}, which will be the basis for our extended formulation of the independence polytope of any regular matroid $M$ that has a star decomposition $M \stackrel{\star}{=} M_0 \oplus_3 M_1 \oplus_3 \cdots \oplus_3 M_k$. Then, we prove Lemma~\ref{lem:star_bound} in case $M_0$ is a graphic matroid. The case where $M_0$ is a cographic matroid will be addressed in the next section.

\subsection{Wong's extended formulation for the spanning tree polytope}\label{sec:stpwong}

Let $D$ be a directed graph, and $r$ any of its nodes, that we call the \emph{root}. An $r$-arborescence of $D$ is an inclusion-wise minimal subset of arcs of $D$ containing, for every node $v$ distinct from $r$ at least one directed path from $r$ to $v$. The \emph{$r$-arborescence polytope} $\arbp(D)$ is the convex hull of the characteristic vectors of $r$-arborescences of $D$. It is well known~\cite[Corollary 52.3a]{schrijver2003combinatorial} that one can express its dominant as
$$
\arbdom(D) = \{c \in \R^{A(D)}_+ \mid \forall S \subsetneq V(D),\ r \in S : c(\delta^{\mathrm{out}}(S)) \geq 1\}
$$
and that $\arbp(D)$ is the face of $\arbdom(D)$ defined by the single valid inequality $c(A(D)) \geq |V(D)| - 1$ (or equivalently, by the valid inequalities $c(\delta^{\mathrm{in}}(v)) \geq 1$ for $v \in V(D) - r$ and $c_a \geq 0$ for $a \in \delta^{\mathrm{in}}(r)$).

The $r$-arborescence dominant has a canonical flow-based compact extended formulation in which every non-root node $v$ has unit flow $\phi^v \in \R^{A(D)}$ from $r$ to $v$, and the variables $c_a$ of the $r$-arborescence dominant act as capacities:
\begin{align}
    \nonumber \arbdom(D) = \Big\{c\in \R^{A(D)} \mid\:& \forall \ v \in V(D)-r : \ \exists \ \phi^v \in \R^{A(D)} :\\
    &\quad \phi^v(\delta^{\mathrm{out}}(r)) - \phi^v(\delta^{\mathrm{in}}(r)) = 1, \label{constr:wongunit}\\
    &\quad \forall \ u \in V(D) - r - v :
    \phi^v(\delta^{\mathrm{out}}(u)) - \phi^v(\delta^{\mathrm{in}}(u)) = 0, \label{constr:wongflowu}\\
    &\quad \mathbf{0} \leq \phi^v \leq c \label{constr:wongcapacity} \Big\}
\end{align}

Let $G$ be a connected graph. Wong's formulation for the spanning tree polytope $\stp(G)$ can be obtained by bidirecting each edge of $G$ to get a directed graph $D$, picking an arbitrary root $r \in V(D)$, and then regarding spanning trees as ``undirected $r$-arborescences''. Formally, 
$$
\stp(G) = \{x \in \R^{E(G)} \mid \exists \ c \in \arbp(D) : \forall \ uv \in E(G) : x_{uv} = c_{(u,v)} + c_{(v,u)}\}\,.
$$
The independence polytope of $M(G)$ can then be expressed as follows:
\begin{align*}
P(M(G)) &= \{x \in \R^{E(G)} \mid \exists \ y \in \stp(G) : 
\mathbf{0} \leq x \leq y\}\\
&= \{x \in \R^{E(G)} \mid \exists \ c \in \arbp(D) : \forall \ uv \in E(G) : 0 \leq x_{uv} \leq c_{(u,v)} + c_{(v,u)}\}\,.
\end{align*}

\subsection{Tweaking Wong's formulation} \label{sec:tweaking_wong}

Assume that $M \stackrel{\star}{=} M_0 \oplus_3 M_1 \oplus_3 \cdots \oplus_3 M_k$ is a 3-connected regular matroid with $M_0$ graphic. Let $G$ be a graph such that $M_0 = M(G)$. One can see (see Section \ref{sec:appendixdecomp}) that $M_0$ is connected, implying that $G$ is connected (actually, even $2$-connected). Let $T_i$ denote the common triangle of $M_0$ and $M_i$ for $i \in [k]$. Hence, $T_1$, \ldots, $T_k$ are (the edge sets of) $k$~edge-disjoint triangles ($3$-cliques) in graph $G$.  

Using as a basis Wong's formulation for $\stp(G)$, we construct an extended formulation for the pair $(P_{T_1,\dots,T_k}(M_0), Q_{T_1,\dots,T_k}(M_0))$ whose size is $O(|V(G)| \cdot |E(G)|)$. That is, we define a polytope $R_{T_1,\dots,T_k}(M_0)$ containing $P_{T_1,\dots,T_k}(M_0)$ and contained in $Q_{T_1,\dots,T_k}(M_0)$ by giving a size-$O(|V(G)| \cdot |E(G)|)$ extended formulation for it.

As before, we partition the edges of $G$ into $T_1$, \ldots, $T_k$ and $E_0 := E(G) \setminus (T_1 \cup \cdots \cup T_k)$. Again, let $D$ denote the directed graph obtained by bidirecting each edge of $G$, and let $r \in V(D)$ be an arbitrary root. For $i \in [k]$, let $B_i := \{(u,v) \mid uv \in T_i\}$ denote the (arc set of the) bidirected triangle obtained from $T_i$. We partition the arcs of $D$ into $B_1$, \ldots, $B_k$ and $A_0 := A(D) \setminus (B_1 \cup \cdots \cup B_k)$.

In addition to the variables $x^0 \in \R^{E_0}$, $x^{T'_1},  x^{T''_1} \in \R^{T_1}$, \ldots, $x^{T'_k}, x^{T''_k} \in \R^{T_k}$, our formulation has 
\begin{itemize}
\item arc capacities $c^0 \in \R^{A_0}$, $c^{T'_1}, c^{T''_1} \in \R^{B_1}$, \ldots, $c^{T'_k}, c^{T''_k} \in \R^{B_k}$,
\item a unit flow $\phi^v \in \R^{A(D)}$ from $r$ to $v$ for each $v \in V(D) - r$,
\item a circulation $\Delta^v_i \in \R^{B_i}$ for each $v \in V(D) - r$ and $i \in [k]$.
\end{itemize}

For each $I \subseteq [k]$ and non-root node $v$, we obtain a unit flow $\phi^v_I \in \R^{A(D)}$ from $r$ to $v$ by adding to $\phi^v$ the circulation $\Delta^v_i$ on the arcs of each $B_i$ with $i \in I$. That is, we let
$$
\phi^v_{I,a} := 
\begin{cases}
\phi^v_a &\text{if } a \in A_0 \text{ or } a \in B_i,\ i \notin I\,,\\
\phi^v_a + \Delta^v_{i,a} &\text{if } a \in B_i,\ i \in I\,.
\end{cases}
$$
The $2^k$ flows $\phi^v_I$ are not explicitly part of the formulation. Instead, they are implicitly defined from the flows $\phi^v$ and the circulations $\Delta^v_i$. The idea is that each circulation $\Delta^v_i$ describes how the flow $\phi^v$ is to be rerouted within the $i$th triangle. 

Now, we give a formal definition of our extended formulation: $R_{T_1,\ldots,T_k}(M_0)$ is the set of tuples $(x^0,x^{T_1'},x^{T_1''},\ldots,x^{T_k'},x^{T_k''})$ such that there exists capacities $(c^0,c^{T_1'},c^{T_1''},\ldots,c^{T_k'},c^{T_k''})$ satisfying the following constraints. First, the $x$-variables and the capacities are related similarly as in the extended formulation for the spanning forest polytope:
\begin{align}
\label{constr:x-UB1} &\forall \ uv \in E_0 : 0 \leq x^0_{uv} \leq c^0_{(u,v)} + c^0_{(v,u)},\\
\label{constr:x-UB2} &\forall \ i \in [k],\ uv \in T_i : 0 \leq x^{T'_i}_{uv} \leq c^{T'_i}_{(u,v)} + c^{T'_i}_{(v,u)},\ 0 \leq x^{T''_i}_{uv} \leq c^{T''_i}_{(u,v)} + c^{T''_i}_{(v,u)}\,.
\end{align}
Second, we include constraints that force $(c^0,c^{T^*_1},\ldots,c^{T^*_k})$ to be in the $r$-arborescence polytope for every choice of $T^*_i \in \{T'_i,T''_i\}$, $i \in [k]$:
\begin{align}
\label{constr:c-arb_eq}
&c^0(A_0) + \sum_{i=1}^k c^{T'_i}(B_i) = |V(D)|-1\,,\\
\label{constr:c-consistency} 
&\forall \ i \in [k] : c^{T'_i}(B_i) = c^{T''_i}(B_i)\,.
\end{align}
Third, for all $v \in V(D) - r$ there exists $\phi^v \in \R^{A(D)}$, $\Delta^v_1 \in \R^{B_1}$, \ldots, $\Delta^v_k \in \R^{B_k}$ such that $\phi^v$ is a unit flow from $r$ to $v$, see \eqref{constr:wongunit} and \eqref{constr:wongflowu} above, and $\Delta^v_i$ is a circulation for all $i \in k$:
\begin{align}
&\forall \ u \in V(D) : \Delta^v_i(\delta^{out}(u) \cap B_i) - \Delta^v_i(\delta^{in}(u) \cap B_i) = 0\,.
\end{align}
Fourth, the flows should satisfy the following lower and upper bounds:
\begin{align}
\label{constr:wongflowbound1}&\forall \ a \in A_0 : 0 \leq \phi^v_a \leq c^0_a\,,\\
\label{constr:wongflowbound2}&\forall \ i \in [k],\ a \in B_i : 0 \leq \phi^v_a \leq c^{T'_i}_a,\ 0 \leq \phi^v_a + \Delta^v_{i,a} \leq c^{T''_i}_a\,.
\end{align}

The resulting formulation has in total $|E_0|+6k$ $x$-variables, $2|E_0|+12k$ $c$-variables, $(|V(G)|-1) \cdot 2|E(G)|$ $\phi$-variables and $(|V(G)|-1)\cdot 6k$ $\Delta$-variables. Given that $|E(G)| = |E_0| + 3k$, the total number of variables is $O(|V(G)| \cdot |E(G)|)$. Since each variable is involved in a constant number of inequalities, the total number of inequalities is also $O(|V(G)| \cdot |E(G)|)$. 

\begin{proposition}\label{prop:wongtweak}
Let $G$ be a connected graph with $k$ edge-disjoint triangles $T_1$, \ldots, $T_k$, and let $M_0 := M(G)$. Letting $R_{T_1,\ldots,T_k}(M_0)$ be defined as above, we have $P_{T_1,\dots,T_k}(M_0) \subseteq R_{T_1,\ldots,T_k}(M_0) \subseteq Q_{T_1,\dots,T_k}(M_0)$.
It follows that $\xc(P_{T_1,\dots,T_k}(M_0),Q_{T_1,\dots,T_k}(M_0)) = O(|V(G)| \cdot |E(G)|)$.
\end{proposition}
\begin{proof}
The inclusion $R_{T_1,\ldots,T_k}(M_0) \subseteq Q_{T_1,\dots,T_k}(M_0)$ follows easily from our construction. Fix any $I \subseteq [k]$, and let $T^*_i := T'_i$ if $i \notin I$ and $T^*_i := T''_i$ if $i \in I$. We see that $(x^0,x^{T^*_1},\ldots,x^{T^*_k})$ is a convex combination of spanning forests since there are capacities $(c^0,c^{T^*_1},\ldots,c^{T^*_k})$ and unit flows $\phi^v_I$ for each $v \in V(D) - r$ that witness this.

In order to prove the inclusion $P_{T_1,\ldots,T_k}(M_0) \subseteq R_{T_1,\dots,T_k}(M_0)$, we only need to focus on the case of spanning trees, since $R_{T_1,\ldots,T_k}(M_0)$ is by definition of anti-blocking type. More precisely, let $J_0 \subseteq E_0$, $J_1', J_1'' \subseteq T_1$, \ldots, $J_k', J_k'' \subseteq T_k$ be such that $J_0 \cup J_1^* \cup \cdots \cup J_k^*$ is a spanning tree of $G$ for all choices of $J_i^* \in \{J_i',J_i''\}$, $i \in [k]$ and in addition $J_i'=J_i''$ or $J_i'=\{\alpha_i,\beta_i\}$ and $J_i''=\{\beta_i,\gamma_i\}$ for all $i \in [k]$. (As before, $\alpha_i$, $\beta_i$ and $\gamma_i$ denote the edges of $T_i$.) We claim that the 0/1-vector $(\chi^{J_0},\chi^{J_1'}, \chi^{J_1''},\dots, \chi^{J_k'}, \chi^{J_k''})$ belongs to $R_{T_1,\dots,T_k}(M_0)$.

We define capacities $(c^0,c^{T'_1},\ldots,c^{T'_k})$ and unit flows $\phi^v$ for $v \in V(D) - r$ from the spanning tree $(J_0,J'_1,J'_2,\ldots,J'_k)$ exactly as in Wong's formulation. For all indices $i \in [k]$ such that $J'_i = J''_i$, we let $c^{T''_i} := c^{T'_i}$ and $\Delta^v_i := \mathbf{0}$ for all $v \in V(D) - r$.

Now consider an index $i \in [k]$ such that $J'_i = \{\alpha_i,\beta_i\}$ and $J''_i = \{\beta_i,\gamma_i\}$. We explain how to define the capacities $c^{T''_i}$ and the ``alternative'' flow values $\phi^v_a + \Delta^v_{i,a}$ for $a \in B_i$ in Figure~\ref{fig:rerouting}. (In case $\phi^v$ is zero on $T_i$, we let $\Delta^v_i := \mathbf{0}$.)

We leave to the reader to check that all constraints defining $R_{T_1,\ldots,T_k}$ are satisfied by our choice of capacities $c^0$ and $c^{T'_i}, c^{T''_i}$ ($i \in [k]$), flows $\phi^v$ ($v \in V(D)-r$) and circulations $\Delta^v_i$ ($v \in V(D)-r$, $i \in [k]$). This establishes the claim, and concludes the proof.
\end{proof}

\begin{figure}
    \centering
    \begin{tabular}{c|c|c|c}
    Capacities &Alt.\ capacities &Flow &Alt.\ flow\\
    \hline
    %
    %
    \begin{tikzpicture}[scale=.35]
        \tikzstyle{vtx}=[circle,draw,thick,inner sep=2.5pt,fill=white]
        \node[vtx] (a) at (90:2)  {};
        \node[vtx] (b) at (210:2) {};
        \node[vtx] (c) at (330:2) {};
        \draw[thick,->,>=latex] (a) -- (b);
        \draw[thick,->,>=latex] (a) -- (c);
    \end{tikzpicture}
    &
    \begin{tikzpicture}[scale=.35]
        \tikzstyle{vtx}=[circle,draw,thick,inner sep=2.5pt,fill=white]
        \node[vtx] (a) at (90:2)  {};
        \node[vtx] (b) at (210:2) {};
        \node[vtx] (c) at (330:2) {};
        \draw[thick,->,>=latex] (a) -- (c);
        \draw[thick,->,>=latex] (c) -- (b);
    \end{tikzpicture}
    &
    \begin{tikzpicture}[scale=.35]
        \tikzstyle{vtx}=[circle,draw,thick,inner sep=2.5pt,fill=white]
        \node[vtx] (a) at (90:2)  {};
        \node[vtx] (b) at (210:2) {};
        \node[vtx] (c) at (330:2) {};
        \draw[thick,->,>=latex] (a) -- (b);
    \end{tikzpicture}
    &
    \begin{tikzpicture}[scale=.35]
        \tikzstyle{vtx}=[circle,draw,thick,inner sep=2.5pt,fill=white]
        \node[vtx] (a) at (90:2)  {};
        \node[vtx] (b) at (210:2) {};
        \node[vtx] (c) at (330:2) {};
        \draw[thick,->,>=latex] (a) -- (c);
        \draw[thick,->,>=latex] (c) -- (b);
    \end{tikzpicture}\\
    \hline
    &&
    \begin{tikzpicture}[scale=.35]
        \tikzstyle{vtx}=[circle,draw,thick,inner sep=2.5pt,fill=white]
        \node[vtx] (a) at (90:2)  {};
        \node[vtx] (b) at (210:2) {};
        \node[vtx] (c) at (330:2) {};
        \draw[thick,->,>=latex] (a) -- (c);
    \end{tikzpicture}
    &
    \begin{tikzpicture}[scale=.35]
        \tikzstyle{vtx}=[circle,draw,thick,inner sep=2.5pt,fill=white]
        \node[vtx] (a) at (90:2)  {};
        \node[vtx] (b) at (210:2) {};
        \node[vtx] (c) at (330:2) {};
        \draw[thick,->,>=latex] (a) -- (c);
    \end{tikzpicture}\\
    %
    %
    \hline
    \begin{tikzpicture}[scale=.35]
        \tikzstyle{vtx}=[circle,draw,thick,inner sep=2.5pt,fill=white]
        \node[vtx] (a) at (90:2)  {};
        \node[vtx] (b) at (210:2) {};
        \node[vtx] (c) at (330:2) {};
        \draw[thick,->,>=latex] (b) -- (a);
        \draw[thick,->,>=latex] (a) -- (c);
    \end{tikzpicture}
    &
    \begin{tikzpicture}[scale=.35]
        \tikzstyle{vtx}=[circle,draw,thick,inner sep=2.5pt,fill=white]
        \node[vtx] (a) at (90:2)  {};
        \node[vtx] (b) at (210:2) {};
        \node[vtx] (c) at (330:2) {};
        \draw[thick,->,>=latex] (b) -- (c);
        \draw[thick,->,>=latex] (c) -- (a);
    \end{tikzpicture}
    &
    \begin{tikzpicture}[scale=.35]
        \tikzstyle{vtx}=[circle,draw,thick,inner sep=2.5pt,fill=white]
        \node[vtx] (a) at (90:2)  {};
        \node[vtx] (b) at (210:2) {};
        \node[vtx] (c) at (330:2) {};
        \draw[thick,->,>=latex] (b) -- (a);
    \end{tikzpicture}
    &
    \begin{tikzpicture}[scale=.35]
        \tikzstyle{vtx}=[circle,draw,thick,inner sep=2.5pt,fill=white]
        \node[vtx] (a) at (90:2)  {};
        \node[vtx] (b) at (210:2) {};
        \node[vtx] (c) at (330:2) {};
        \draw[thick,->,>=latex] (b) -- (c);
        \draw[thick,->,>=latex] (c) -- (a);
    \end{tikzpicture}\\
    \hline
    &&
    \begin{tikzpicture}[scale=.35]
        \tikzstyle{vtx}=[circle,draw,thick,inner sep=2.5pt,fill=white]
        \node[vtx] (a) at (90:2)  {};
        \node[vtx] (b) at (210:2) {};
        \node[vtx] (c) at (330:2) {};
        \draw[thick,->,>=latex] (b) -- (a);
        \draw[thick,->,>=latex] (a) -- (c);
    \end{tikzpicture}
    &
    \begin{tikzpicture}[scale=.35]
        \tikzstyle{vtx}=[circle,draw,thick,inner sep=2.5pt,fill=white]
        \node[vtx] (a) at (90:2)  {};
        \node[vtx] (b) at (210:2) {};
        \node[vtx] (c) at (330:2) {};
        \draw[thick,->,>=latex] (b) -- (c);
    \end{tikzpicture}\\
    \hline
     %
    %
    \begin{tikzpicture}[scale=.35]
        \tikzstyle{vtx}=[circle,draw,thick,inner sep=2.5pt,fill=white]
        \node[vtx] (a) at (90:2)  {};
        \node[vtx] (b) at (210:2) {};
        \node[vtx] (c) at (330:2) {};
        \draw[thick,->,>=latex] (c) -- (a);
        \draw[thick,->,>=latex] (a) -- (b);
    \end{tikzpicture}
    &
    \begin{tikzpicture}[scale=.35]
        \tikzstyle{vtx}=[circle,draw,thick,inner sep=2.5pt,fill=white]
        \node[vtx] (a) at (90:2)  {};
        \node[vtx] (b) at (210:2) {};
        \node[vtx] (c) at (330:2) {};
        \draw[thick,->,>=latex] (c) -- (b);
        \draw[thick,->,>=latex] (c) -- (a);
    \end{tikzpicture}
    &
    \begin{tikzpicture}[scale=.35]
        \tikzstyle{vtx}=[circle,draw,thick,inner sep=2.5pt,fill=white]
        \node[vtx] (a) at (90:2)  {};
        \node[vtx] (b) at (210:2) {};
        \node[vtx] (c) at (330:2) {};
        \draw[thick,->,>=latex] (c) -- (a);
    \end{tikzpicture}
    &
    \begin{tikzpicture}[scale=.35]
        \tikzstyle{vtx}=[circle,draw,thick,inner sep=2.5pt,fill=white]
        \node[vtx] (a) at (90:2)  {};
        \node[vtx] (b) at (210:2) {};
        \node[vtx] (c) at (330:2) {};
        \draw[thick,->,>=latex] (c) -- (a);
    \end{tikzpicture}\\
    \hline
    &&
    \begin{tikzpicture}[scale=.35]
        \tikzstyle{vtx}=[circle,draw,thick,inner sep=2.5pt,fill=white]
        \node[vtx] (a) at (90:2)  {};
        \node[vtx] (b) at (210:2) {};
        \node[vtx] (c) at (330:2) {};
        \draw[thick,->,>=latex] (c) -- (a);
        \draw[thick,->,>=latex] (a) -- (b);
    \end{tikzpicture}
    &
    \begin{tikzpicture}[scale=.35]
        \tikzstyle{vtx}=[circle,draw,thick,inner sep=2.5pt,fill=white]
        \node[vtx] (a) at (90:2)  {};
        \node[vtx] (b) at (210:2) {};
        \node[vtx] (c) at (330:2) {};
        \draw[thick,->,>=latex] (c) -- (b);
    \end{tikzpicture}\\
    \end{tabular}
    \caption{Definition of $c^{T''_i}$ (second column) and $\phi^v + \Delta^v_i$ (fourth column), in each case.}
    \label{fig:rerouting}
\end{figure}
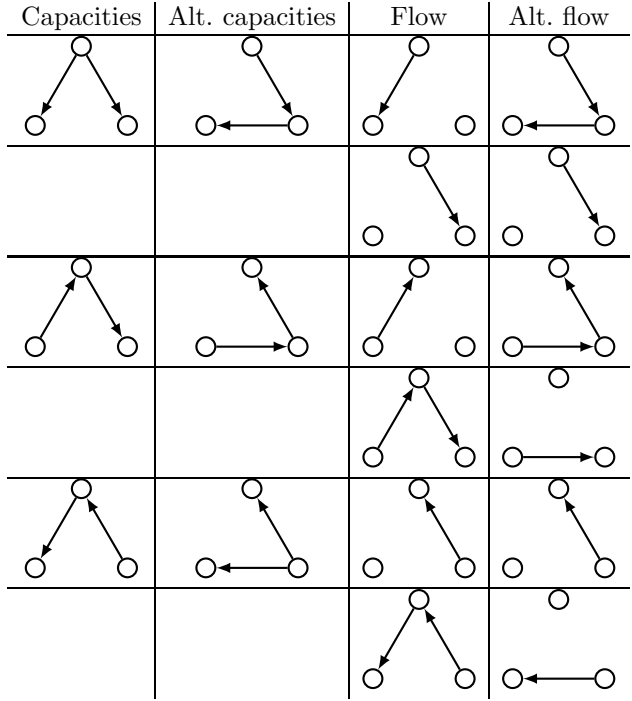

\section{Smaller formulation for star decompositions: the cographic case} \label{sec:cographic}

In this section, we consider the case where $M_0$ is cographic, i.e.\ $M_0=M^*(G)$ for a ($2$-)connected graph $G$. Our goal is to prove Lemma~\ref{lem:star_bound} in this case. As in the previous section, we rely on Proposition~\ref{prop:asymmetricstar}.

By duality, we have that $x \in B(M_0)$ if and only if $\textbf{1}-x \in B(M_0^*)$. Hence, we will again deal with the spanning tree polytope of $G$. If $T$ is a triangle of cographic matroid $M_0$, then the corresponding edges of $G$ form a cut of size three. Let $T_1,\dots, T_k$ be the pairwise disjoint triangles of $M_0$ involved in the $3$-sums with $M_1,\dots,M_k$ respectively, where $T_i = \{\alpha_i,\beta_i,\gamma_i\}$ for $i \in [k]$, as previously. We can assume (see the appendix, and specifically Proposition \ref{prop:claws}) that each $T_i$ is of the form $\delta(v_i)$ for some degree-$3$ node $v_i \in V(G)$. We denote by $a_i$, $b_i$, and $c_i$ the neighbors of $v_i$. We may assume that $\alpha_i = a_iv_i$, $\beta_i = b_iv_i$ and $\gamma_i = c_iv_i$.

Let $V^* := \{v_1,\dots,v_k\}$. Observe that $V^*$ is a stable set of $G$. Let $D$ be the directed graph obtained from $G$ by bidirecting each edge. Let $r$ be any node in $V(D) \setminus V^*$. Such a node exists since we can take $r := a_1$ for instance.

As a first step, we simplify Wong's formulation for the $r$-arborescence polytope of $D$. As stated in the next lemma, it is sufficient to have a unit flow $\phi^v$ for each $v \in V(G) \setminus (V^* \cup \{r\})$, and impose a single specific constraint on the arcs entering $v_i$ for each $i \in [k]$. 

\begin{lemma}\label{lem:wongmodcographic}
  Let $D$ be a directed graph with specified distinct node $r, v_1, \dots, v_k$ (for some $k\geq 1$) such that $\delta^{\mathrm{out}}(v_i) \cap \delta^{\mathrm{in}}(v_j) = \emptyset$ for every $i, j \in [k]$. Letting $V^* := \{v_1,\ldots,v_k\}$, we have
  \begin{align*}
    \nonumber \arbdom(D) = \Big\{c\in \R^{A(D)} \mid\:&
    \forall i \in [k] : c(\delta^{\mathrm{in}}(v_i)) \geq 1,\\
    &\forall v \in V(D) \setminus (V^* \cup \{r\}) : \exists \phi^v \in \R^{A(D)} : \text{\eqref{constr:wongunit}--\eqref{constr:wongcapacity}} \Big\}\,.
\end{align*}
\end{lemma}
\begin{proof}
Let $Q(D)$ denote the right-hand side of the target equation. Proving that $\arbdom(D) \subseteq Q(D)$ is  straightforward. Let $A \subseteq A(D)$ be an $r$-arborescence, and let $c := \chi^A$. For each $v \in V(D) \setminus (V^* \cup \{r\})$, define $\phi^v$ as the characteristic vector of the $r$--$v$ directed path in $A$. All constraints of $Q(D)$ are clearly satisfied by this choice of unit flows.

Now, we prove $Q(D) \subseteq \arbdom(D)$ by using the linear description of $\arbdom(D)$, see Section \ref{sec:stpwong}. Let $c \in Q(D)$ and let $\phi^v$, $v \in V(D) \setminus (V^* \cup \{r\})$ be corresponding flows. 

Consider any proper node subset $S$ with $r \in S$. If there is any $v \in V(D) \setminus S$ with $v \notin V^*$, then we have $c(\delta^{\mathrm{out}}(S)) \geq \phi^v(\delta^{\mathrm{out}}(S)) \geq 1$, since $S$ is an $r$--$v$ cut and $\phi^v$ is an $r$--$v$ flow of value~$1$.  Otherwise, $V(D) \setminus S \subseteq V^*$. Pick an arbitrary node $v_i \in V^* \cap (V(D) \setminus S)$. 
Since $\delta^{\mathrm{in}}(v_i) \subseteq \delta^{\mathrm{out}}(S)$, we have $c(\delta^{\mathrm{out}}(S)) \geq c(\delta^{\mathrm{in}}(v_i)) \geq 1$.
\end{proof}

We are now ready to describe the extended formulation for our intermediate polytope $R_{T_1,\ldots,T_k}(M_0)$. The formulation is similar to that given in the previous section for the graphic case, except that our starting point is the formulation for $\arbdom(D)$ given in Lemma \ref{lem:wongmodcographic}. Also, it turns out that we do not need the $\Delta$-variables. Finally, the root node $r$ should be picked outside of $V^* := \{v_1,\ldots,v_k\}$.

Using the same notation as above in Section~\ref{sec:tweaking_wong}, we define $R_{T_1,\ldots,T_k}(M_0)$ as the set of tuples $(x^0,x^{T_1'},x^{T_1''}$, \ldots,  
$x^{T_k'},x^{T_k''})$ such that there exist capacities $(c^0,c^{T_1'},c^{T_1''},\ldots,c^{T_k'},c^{T_k''})$ satisfying the following constraints. First, instead of~\eqref{constr:x-UB1} and \eqref{constr:x-UB2} we ask
\begin{align}
    &\forall uv \in E_0 : 0 \leq x^0_{uv} \leq 1 - c^0_{(u,v)} - c^0_{(v,u)}\,,\\
    &\forall i \in [k],\ uv \in T_i : 0 \leq x^{T'_i}_{uv} \leq 1 - c^{T'_i}_{(u,v)} - c^{T'_i}_{(v,u)},\ 
    0 \leq x^{T''_i}_{uv} \leq 1 - c^{T''_i}_{(u,v)} - c^{T''_i}_{(v,u)}\,.
\end{align}

Second, we impose constraints~\eqref{constr:c-arb_eq} and \eqref{constr:c-consistency} as before. 

Third, for all $v \in V(D) \setminus (V^* \cup \{r\})$ there exists $\phi^v \in \R^{A(D)}$ satisfying \eqref{constr:wongunit} and \eqref{constr:wongflowu}. 

Fourth, the flows $\phi^v$ should satisfy the bounds \eqref{constr:wongflowbound1} and 
\begin{align}
&\forall \ i \in [k],\ a \in B_i : 0 \leq \phi^v_a \leq c^{T'_i}_a,\ \phi^v_a \leq c^{T''_i}_a\,    
\end{align}
(this last constraint replaces~\eqref{constr:wongflowbound2}).

Fifth, we include explicit constraints on the capacities entering each node in $V^*$, as in Lemma~\ref{lem:wongmodcographic}:
\begin{align}
\label{constr:c-bound_explicit}&\forall i \in [k] : c^{T'_i}(\delta^{\mathrm{in}}(v_i)) \geq 1,\
c^{T''_i}(\delta^{\mathrm{in}}(v_i)) \geq 1
\end{align}

One can easily check that the extended formulation defining $R_{T_1,\ldots,T_k}(M_0)$ has size $O(|V(G)| \cdot |E(G)|)$.

\begin{proposition}\label{prop:wongtweakcographic}
Let $G$ be connected graph, let $V^* := \{v_1,\ldots,v_k\}$ be a nonempty stable set such that each $v_i$ has degree~$3$, let $T_i := \delta(v_i)$ for $i \in [k]$, and let $M_0 := M^*(G)$ be the cographic matroid associated with $G$. Letting $R_{T_1,\ldots,T_k}(M_0)$ be defined as above, we have $P_{T_1,\dots,T_k}(M_0) \subseteq R_{T_1,\ldots,T_k}(M_0) \subseteq Q_{T_1,\dots,T_k}(M_0)$. Hence, $\xc(P_{T_1,\dots,T_k}(M_0),Q_{T_1,\dots,T_k}(M_0)) = O(|V(G)| \cdot |E(G)|)$.
\end{proposition}

\begin{proof}
Let $x := (x^0,x^{T'_1},x^{T''_1},\ldots,x^{T'_k},x^{T'_k}) \in R_{T_1,\ldots,T_k}(M_0)$, let $(c^0,c^{T'_1},c^{T''_1},\ldots,c^{T'_k},c^{T''_k})$ be capacities and let $\phi^v$, $v \in V(D) \setminus (V^* \cup \{r\})$ be unit flows witnessing $x \in R_{T_1,\ldots,T_k}(M_0)$. It should be clear from Lemma~\ref{lem:wongmodcographic} and \eqref{constr:c-arb_eq} and \eqref{constr:c-consistency} that $(c^0,c^{T*_1},\ldots,c^{T^*_k})$ is in the $r$-arborescence polytope for every choice of $T^*_i \in \{T'_i,T''_i\}$, $i \in [k]$. Hence, $(x^0,x^{T*_1},\ldots,x^{T^*_k}) \in P(M_0)$ for every choice of $T^*_i$. Therefore, $x \in Q_{T_1,\ldots,T_k}(M_0)$. This proves the rightmost inclusion.

In order to prove the leftmost inclusion, let $J_0 \subseteq E_0$, $J_1', J_1'' \subseteq T_1$, \ldots, $J_k', J_k'' \subseteq T_k$ be such that $J_0 \cup J_1^* \cup \cdots \cup J_k^*$ is a basis of $M_0$ (that is, the complement of a spanning tree of $G$) for all choices of $J_i^* \in \{J_i',J_i''\}$, $i \in [k]$ and in addition $J_i'=J_i''$ or $J_i'=\{\alpha_i,\beta_i\}$ and $J_i''=\{\beta_i,\gamma_i\}$ for all $i \in [k]$. 

Let $x := (\chi^{J_0},\chi^{J_1'}, \chi^{J_1''},\dots, \chi^{J_k'}, \chi^{J_k''})$. The capacities $(c^0,c^{T'_1},\ldots,c^{T'_k}) \in \R^{A(D)}$ are simply the characteristic vector of the $r$-arborescence obtained from the complement of $J_0 \cup J'_1 \cup \ldots \cup J'_k$. In case $J_i'' = J_i'$ we let $c^{T''_i} := c^{T'_i}$. Otherwise, $J'_i = \{\alpha_i,\beta_i\}$ and $J''_i = \{\beta_i,\gamma_i\}$, which means in particular that $v_i$ is a leaf of all spanning trees $E(G) \setminus (J_0 \cup J_1^* \cup \cdots \cup J_k^*)$. Also, $c^{T'_i} = \mathbf{e}_{(c_i,v_i)}$. We define $c^{T''_i} \in \R^{B_i}$ by letting $c^{T''_i} := \mathbf{e}_{(a_i,v_i)}$.

Again, all the constraints defining $R_{T_1,\ldots,T_k}(M_0)$ are satisfied by this choice of capacities $(c^0,c^{T'_1},c^{T''_1}$, \ldots, $c^{T'_k},c^{T''_k})$ and flows $\phi^v$ for $v \in V(D) \setminus (V^* \cup \{r\})$. \end{proof}

We are now ready to prove Lemma~\ref{lem:star_bound}.

\begin{proof}[Proof of Lemma~\ref{lem:star_bound}]
The bound on $\xc(P(M))$ follows directly either from Propositions~\ref{prop:asymmetricstar} and \ref{prop:wongtweak} in case $M_0$ is graphic, or from  Propositions~\ref{prop:asymmetricstar} and \ref{prop:wongtweakcographic} in case $M_0$ is cographic.
\end{proof}

\section{The circuit dominant of a regular matroid}\label{sec:cdom}

In this section, we deal with another polyhedron that one can define for every matroid $M$, and provide a small extended formulation for it whenever $M$ is regular. 

Recall that the circuit dominant of a matroid $M$ is $\cdom(M) = \conv\{\chi^C \mid C $ is a circuit of $M\}+\R^{E(M)}_+$. When $M$ is cographic, $\cdom(M)$ is known as the \emph{cut dominant}, for which several polynomial size extended formulations are known \cite{tamir1994polynomial, weltge2015sizes}. In this section, we extend this to all regular matroids. We stress that, even for the special case of cographic matroids, a complete facial description of the circuit dominant in the original space is not known (see \cite{conforti2016cut, conforti2004cut}).

We remark that, in \cite{kapadia2014}, a polynomial time algorithm is given to find a minimum weight circuit in a regular matroid (under the assumption that the weights are non-negative). The algorithm uses Seymour's decomposition theorem, which suggests that a small extended formulation for $\cdom(M)$ could be found using similar techniques as done in this paper for the independence polytope. However, we give here a direct extended formulation that is based on total unimodularity only. Given its compactness, our extended formulation together with any polynomial time algorithm for linear programming provides an alternative way to find minimum weight circuits in regular matroids, that is arguably simpler than the one given in~\cite{kapadia2014}. 

Before establishing the main result of this section, we need the following simple fact on cycles of regular matroids. 

\begin{proposition}\label{prop:circuitvector}
Let $M$ be a regular matroid on ground set $E$ and let $A$ be a TU (totally unimodular) matrix representing $M$ over $\R$. Let $C \subseteq E$. Then $C$ is a cycle of $M$ if and only if there is a vector $\psi = \psi(C) \in \{-1,0,1\}^{E}$ whose support is $C$ and such that $A \psi =\mathbf{0}$.
\end{proposition}
\begin{proof}
First, we remark that, since $A$ is TU, $A$ represents $M$ over any field, in particular over $\mathbb{F}_2$. Hence, the ``if" direction follows. Indeed, suppose that a vector $\psi \in \{-1,0,1\}^{E}$ with support $C$ satisfies $A \psi = \mathbf{0}$. Then, taking this equation over $\mathbb{F}_2$, we see that $A \chi^C = \mathbf{0}$, implying that $C$ is a cycle of $M$. 

For the ``only if" direction, we can restrict $C$ to be a circuit. Let $\varphi \in \R^{E}$ be any vector such that $A \varphi = \mathbf{0}$ and the support of $\varphi$ equals $C$. Consider the polytope $P = P(C) := \{x \in \R^{E} \mid Ax = \mathbf{0},\ \forall e \in C : -1 \leqslant x_e \leqslant 1,\ \forall e \in E \setminus C : x_e = 0\}$. Since $A$ is TU, $P$ is integer. In fact, the vertex set of $P$ is contained in $\{-1,0,1\}^{E}$. Since, $\varphi / ||\varphi||_\infty$ is a non-zero point in $P$, we see that $P$ has a vertex that is non-zero. Let $\psi$ be any such vertex. Since $C$ is a circuit, $\psi$ satisfies all the required properties.
\end{proof}

\begin{theorem}\label{thm:cdom}
Let $M$ be a regular matroid on ground set $E$, represented (over $\R$) by a totally unimodular matrix $A$. For $e\in E$, define the polyhedron 
\begin{align}  
\nonumber \cdom(M,e) := \{x \mid \exists y : -x \leqslant y &\leqslant x,\\
Ay &=\mathbf{0},\label{eq:cdomfirst}\\
y_e &=1,\\
-\mathbf{1} \leqslant y &\leqslant \mathbf{1}\}\,.\label{eq:cdomlast}
\end{align}
Then $\cdom(M)=\conv \left( \bigcup_{e\in E}\cdom(M,e) \right)$. In particular, $\xc(\cdom(M)) = O(|E|^2)$.
\end{theorem}
\begin{proof}
It suffices to show, for each $e\in E$, that
$$
\cdom(M,e)=\conv\{\chi^C \mid C \text{ is a circuit of } M, e\in C\}+\R^E_+\,.
$$
Fix an arbitrary element $e \in E$.

For the ``$\supseteq$" inclusion, let $C$ be a circuit of $M$ with $e\in C$. Then $\chi^C \in \cdom(M,e)$ by setting $y := \pm \psi(C)$ (see Proposition \ref{prop:circuitvector}). The inclusion follows since $\cdom(M,e)$ is upward closed. 

For the ``$\subseteq$" inclusion, we show that any vertex of $\cdom(M,e)$ dominates $\chi^C$, for some circuit $C$ of $M$ containing $e$.
Consider the polytope $Q = Q(e) := \{y \mid \text{\eqref{eq:cdomfirst}--\eqref{eq:cdomlast}}\}$ and the polyhedron $R = R(e) := \{(x,y) \mid -x \leqslant y \leqslant x,\ y \in Q\}$. 

Fix a vertex $x^*$ of $\cdom(M,e)$, and let $I_0 := \{i \mid x^*_i = 0\}$. There exists $y^*\in Q$ such that $(x^*,y^*)$ is a vertex of $R$. Notice that $y^*_i = 0$ for all $i \in I_0$. For $i \notin I_0$, we have either $x^*_i = y^*_i$ or $x^*_i = -y^*_i$, but not both. Necessarily, $y^*$ is a vertex of the polytope $Q \cap \{y \mid \forall i \in I_0 : y_i = 0\}$. Since $Q$ is defined by a TU system with integral right-hand sides, its intersection with any coordinate subspace is integral. Hence $y^*$ is an integer point, and so is $(x^*,y^*)$. 

Let $D$ denote the support of $y^*$, so that $x^* = \chi^D$. By Proposition \ref{prop:circuitvector}, $D$ is a cycle of $M$. Moreover, $D$ contains $e$. Thus, $D$ contains a circuit $C$ of $M$ containing $e$. Since $x^*$ dominates $\chi^C$, we are done.

The bound on $\xc(\cdom(M))$ now follows from Balas Union (see Proposition \ref{prop:Balas}), a version of which applies to unbounded polyhedra with the same recession cone \cite{balas1979disjunctive}.
\end{proof}

We conclude the section by remarking that, due to the fact that the dual of a regular matroid is regular, Theorem \ref{thm:cdom} applies to the \emph{cocircuit} dominant as well (which is defined in the obvious way).

\section{Discussion} \label{sec:discussion}

It is straightforward to improve the $O(n^6)$ bound of Theorem~\ref{thm:main} to a $O(n^{6-\varepsilon})$ bound, for sufficiently small $\varepsilon > 0$ (for instance, we may take $\epsilon = .41$). However, we believe that a better bound should hold. We leave this as our first open problem. 

Related to this question, we suspect that the simple upper bound $\xc(P(M_1 \oplus_3 M_2)) \leq \xc(P(M_1)) + \xc(P(M_2))$ \emph{fails} for some regular matroids $M_1$ and $M_2$, although we do not have any concrete counterexample. If the simple bound held, then this would give a $O(n^2)$ upper bound on $\xc(P(M))$ for all regular matroids $M$ on $n$ elements, see~\cite{kaibel2016extended,weltge2015sizes}.

Weltge~\cite{weltge2015sizes} proved that for every graphic matroid $M = M(G)$, the extension complexity of $\cdom(M^*)$ (that is, the cut dominant of $G$)  is bounded as follows:
\begin{equation}
\label{eq:Weltge}
\xc(\cdom(M^*)) \leqslant \xc(P(M))+O(|E(M)|)\,.
\end{equation}
We do not know whether $\xc(P(M))$ can be similarly bounded in terms of $\xc(\cdom(M^*))$. If this could be done for all regular matroids $M$, then this could lead to an improved upper bound on the extension complexity of $P(M)$, via Theorem~\ref{thm:cdom}. 


Rothvoss~\cite{rothvoss2013some} has proved via a counting argument involving \emph{sparse paving} matroids that the independence polytope of many matroids has exponential extension complexity. It is unclear that one can find an explicit infinite family of sparse paving matroids $M$ with $\xc(P(M))$ superpolynomial, since that would automatically yield an explicit infinite family of Boolean functions requiring superlogarithmic depth circuits, see G\"o\"os~\cite{Goos16}.

At this point, we do not know what is the worst case extension complexity of $P(M)$ even when $M$ is a binary matroid. Let $f(n)$ denote the maximum of $\xc(P(M))$ where $M$ is a binary matroid on $n$ elements. Is $f(n)$ polynomial? This is our second open problem. 

We stress that optimizing over $\cdom(M)$ in a general binary matroid is NP-hard~\cite{vardy1997intractability}. We suspect that $\xc(\cdom(M))$ is superpolynomial for these matroids. If \eqref{eq:Weltge} could be generalized to all binary matroids (even with worse polynomial bounds), then this would give explicit binary matroids with $\xc(P(M))$ superpolynomial.

We conclude by remarking that, by Proposition 3.3.10 of \cite{weltge2015sizes}, Theorem \ref{thm:main} implies a polynomial bound on the extension complexity of independence polytopes of almost-regular matroids.

\section{Acknowledgements}

The first author thanks Georg Loho, Volker Kaibel, Matthias Walter and Stefan Weltge for joining the first attempts to solve the flaw in~\cite{kaibel2016extended}. We also thank Tony Huynh for taking part in the early stages of the research. Finally, we thank Stefan Weltge for discussions related to Section~\ref{sec:cdom}. This project was supported by ERC Consolidator Grant 615640-ForEFront.

\bibliographystyle{plain}
\bibliography{regular-bib}

\appendix 

\section{Proofs missing from Section~\ref{sec:asymmetric}}

\begin{proof}[Proof of Lemma~\ref{lem:3sumbases}]
Let $B$ be a basis of $M$. Thanks to Lemma~\ref{lem:3sumfacts}, $|B|=\rk(M_1)+\rk(M_2)-2$, and $B\cap E(M_i)$ is independent in $M_i$ for all $i \in [2]$. Without loss of generality, assume that $|B\cap E(M_1)| \geq |B\cap E(M_2)|$. Two cases are possible. \medskip

\noindent (i) $|B\cap E(M_1)|=\rk(M_1)$ and $|B \cap E(M_2)| = \rk(M_2)-2$. Hence, $B_1 := B\cap E(M_1)$ is a basis of $M_1$ and $B \cap E(M_2)$ is an independent set of $M_2$ with $\rk(M_2)-2$ elements. Then, there are two (distinct) elements $t_1, t_2 \in T$ such that $B_2 := (B \cap E(M_2)) + t_1 + t_2$ is a basis of $M_2$.\medskip

\noindent (ii) $|B\cap E(M_1)|=\rk(M_1)-1$ and $|B \cap E(M_2)|=\rk(M_2)-1$. Since $T$ is a triangle, there are two choices for $t \in T$ such that $(B \cap E(M_1)) + t$ is a basis of $M_1$ and two choices for $t \in T$ such that $(B \cap E(M_2)) + t$ is a basis of $M_2$. Moreover, these sets of choices are distinct since otherwise there is $t' \in T$ such that $(B \cap E(M_1)) + t'$ and $(B \cap E(M_2)) + t'$ are both dependent. But then there is a circuit $C_1$ of $M_1$ contained in $(B \cap E(M_1)) + t'$ and a circuit $C_2$ of $M_2$ contained in $(B \cap E(M_2)) + t'$ that both intersect $T$ in the single element $t'$. We see that $C_1 \Delta C_2$ is a cycle of $M$ contained in $B$, a contradiction.

Now, we prove the backward implication. 
Notice that, in both cases, $B$ has the cardinality of a basis of $M$. Towards a contradiction, assume that $B$ is not a basis. Then $B$ contains a circuit $C = C_1 \Delta C_2$ where $C_1$ and $C_2$ are cycles in $M_1$ and $M_2$ respectively, such that $C_1 \cap T = C_2 \cap T$. None of the cycles $C_1$ and $C_1 \Delta T$ can be (non-empty and) contained in $B_1$. Similarly, none of the cycles $C_2$ and $C_2 \Delta T$ can be (non-empty and) contained in $B_2$. In all cases, we get a contradiction.
\end{proof}

\begin{proof} [Proof of Lemma~\ref{lem:3sumflats}]
Thanks to Lemma~\ref{lem:3sumfacts}, and to the fact that $|F\cap T|\in\{0,1,3\}$ for a flat $F$ and a triangle $T$, we only need to prove the statements about connectedness and the rank.

First consider the case $F \subseteq E(M_i)$ for some $i \in [2]$. We only need to show that $F$ is connected in $M_i$ too. This follows for instance from the fact that cycles of $M$ contained in $F$ are exactly the cycles of $M_i$ contained in $F$. 

Now assume that $F \cap E(M_i) \neq \emptyset$ for all $i \in [2]$. From Lemma~\ref{lem:3sumfacts}, there are flats $F_1$, $F_2$ of $M_1$, $M_2$ respectively such that $F=F_1\Delta F_2$. Since $F$ is connected, we have $F_1\cap T=F_2\cap T\neq \emptyset$: indeed, consider any circuit $C$ contained in $F$ intersecting both $E_1, E_2$, it satisfies $C=C_1\Delta C_2$, with $C_i$ cycle of $M_i$ and $|C_i\cap T|=1$, for $i=1,2$, but then since $C_i\setminus T\subset F$ we must have $C_i\subset F$ as $F$ is a flat. For a similar reasoning on the circuits of $F$, we have that $F_1,F_2$ are connected. Hence we are left with two cases, according to the size of $F_i\cap T$. If $|F_i\cap T|=1$, say $F_i\cap T=\alpha$, $i=1,2$, then we claim that $M|F=M_1|F_1\oplus_2 M_2|F_2$, which implies the statement on the rank. Indeed, by definition of restriction and of 3-sum, the cycles of $M|F$ are the cycles of $M$ that are contained in $F$, and they are exactly the cycles of $M_i$ contained in $F_i$ for $i=1$ or 2, or have form $C=C_1\Delta C_2$, with $C_i$ cycle of $M_i$ and $C_i\cap T=\alpha$. The latter is the definition of $2$-sum.

Finally, if $|F_i\cap T|=3$, then arguing exactly as above we can show that $M|F=M_1|F_1\oplus_3 M_2|F_2$, which implies the statement on the rank.
\end{proof}

\section{Proving Theorem \ref{thm:seymourdecomposition3conn}}\label{sec:appendixdecomp}
In this section we elaborate on Seymour's decomposition theorem, and prove Theorem~\ref{thm:seymourdecomposition3conn}.

First, we introduce some concepts of matroid theory that will be useful in the following.
A \emph{k-separation} of a matroid on ground set $E$ is a partition $(A,B)$ of $E$ with $|A|,|B|\geq k$, and $\rk(A)+\rk(B)\leq \rk(E)+k-1$. The separation is said to be \emph{exact} if equality holds.
$k$-separations are intimately connected to $k$-sums, for $k=1,2,3$. The following is well known (see \cite{oxley2006matroid}, page 421).

\begin{lemma}\label{prop:1,2separation} 
  A matroid $M$ is a 1-sum if and only if it has a 1-separation, and is a 2-sum if and only if it has an exact 2-separation.
\end{lemma}

Notice that, although the previous lemma applies to all matroids, we are only concerned with binary matroids.

In light of Lemma \ref{prop:1,2separation}, we have that a matroid is connected if and only if it has no 1-separation, and 3-connected if and only if it has no 2-separation or 1-separation.
We will be mainly concerned with the case $k=3$, which is slightly more delicate. For simplicity, we will call \emph{non-trivial} an exact 3-separation $(A,B)$ with $|A|, |B|\geq 4$. The next lemma essentially states that for a binary matroid being a 3-sum is equivalent to having a non-trivial 3-separation.

\begin{lemma}[Proposition 12.4.17 of \cite{oxley2006matroid}]\label{prop:3separation}
  If a binary matroid $M$ has a non-trivial 3-separation $(A,B)$ then there are two binary matroids $M_1$, $M_2$ with $E(M_1) = A\cup T$, $E(M_2) = B\cup T$, where $T$ is a triangle of both $M_1$ and $M_2$, such that $M=M_1\oplus_3 M_2$. On the other hand, if $M$, $M_1$, $M_2$ are binary matroids such that $M=M_1\oplus_3 M_2,$ then $(E(M) \cap E(M_1), E(M) \cap E(M_2))$ is a non-trivial 3-separation of $M$.
\end{lemma}

A matroid $N$ is a \emph{minor} of a matroid $M$ if $N=M/X\setminus Y$ for some $X, Y$, i.e.\ if $N$ can be obtained from $M$ through a sequence of deletions and contractions. The class of regular matroids (as many others) is closed under taking minors. It is not hard to show that, if $M$ is the 1- or 2-sum of $M_1, M_2$, then $M_1, M_2$ are isomorphic to minors of $M$. One of the main results of \cite{seymour1980decomposition} states that the same is true if $M$ is the 3-sum of $M_1, M_2$, provided that $M$ is 3-connected. Moreover, in \cite{seymour1980decomposition} two special regular matroids are defined, $R_{10}$ and $R_{12}$ on 10 and 12 elements respectively, with the following properties.

\begin{lemma}\label{prop:R10,R12}
 \begin{enumerate}
     \item If a regular matroid (different from $R_{10}$) contains $R_{10}$ as a minor, then it has a 1- or 2-separation.
     \item If a regular matroid has $R_{12}$ as a minor, then it has a non-trivial 3-separation.
     \item If a matroid is regular, 3-connected, and is not graphic, cographic or $R_{10},$ then it has $R_{12}$ as a minor.
 \end{enumerate} 
\end{lemma}

This gives a way to iteratively decompose a regular matroid $M$: if it is not already a ``basic" matroid, i.e.\ graphic, cographic, or isomorphic to $R_{10}$, then $M$ has either a 1- or 2- separation, or a non-trivial 3-separation, hence it can be expressed as a 1-, 2-, or 3-sum of some smaller matroids. Such matroids are minors of $M$, hence are regular, and can be further decomposed until all matroids obtained are basic. 
   
The resulting decomposition process can be described by a ``decomposition tree" in a natural way, where the nodes are basic matroids and the edges represent the operations between them. However, for simplicity we would like to consider a decomposition which involves 3-sums only, as described in Theorem~\ref{thm:seymourdecomposition3conn}. 

For this reason we start from a 3-connected regular matroid $M$ that is not $R_{10}$. From Lemma~\ref{prop:R10,R12} such $M$ cannot have $R_{10}$ as a minor, hence none of its minors (in particular the matroids that we will meet during the decomposition process) can. If $M$ is not graphic or cographic, we write it as $M = M_1 \oplus_3 M_2$. Now, consider $M_i$ for $i \in [2]$. If it is graphic or cographic, we can stop decomposing. Otherwise, we can decompose $M_i$  further as a $k$-sum for some $k \in [3]$. If $M_i$ is not 3-connected, we might need to decompose it as a 1- or 2-sum. We argue that this never happens, thanks to the following lemma. Recall that two elements of a matroid are \emph{parallel} if they form a circuit of size 2.

\begin{lemma}[(4.3) in \cite{seymour1980decomposition}]\label{prop:not3conn}
  Suppose that $M$ is a 3-connected binary matroid, and $M=M_1\oplus_3 M_2$, where $M_1, M_2$ share a triangle $T$. If $(A,B)$ is a 2-separation of $M_i$ for $i \in [2]$ with $|A|\leq |B|$, then $A$ consists of two parallel elements $\alpha,\alpha'$, with $\alpha\in T$ and $\alpha' \not\in T$.
\end{lemma}

This implies that, although $M_i$ might not be 3-connected, it is close to being 3-connected: deleting repeated elements results in a 3-connected matroid. Recall that the \emph{simplification} $\si(M)$ is the matroid obtained from $M$ by removing all loops and deleting elements until no two elements are parallel. Clearly, $\rk(M)=\rk(\si(M))$. We say that a binary matroid is \emph{almost 3-connected} if its simplification is 3-connected. 

\begin{lemma}\label{lem:almost3con}
Let $M$ be an almost $3$-connected regular matroid that is not graphic, cographic and has no minor isomorphic to $R_{10}$. Then $M = M_1 \oplus_3 M_2$, where $M_1, M_2$ are isomorphic to minors of $M$, and are almost $3$-connected.  
\end{lemma}
\begin{proof}
As adding parallel elements to a (co)graphic matroid leaves it (co)graphic, we have that $\si(M)$ is not graphic or cographic. Moreover, $\si(M)$ cannot have $R_{10}$ as a minor. Hence, applying parts 2 and 3 of Lemma \ref{prop:R10,R12}, we have that $\si(M)$ has a non-trivial 3-separation $(A',B')$, corresponding to a 3-sum $\si(M)=M_1'\oplus_3 M_2'$. As $\si(M)$ is 3-connected, one can see, as a simple consequence of Lemma \ref{prop:not3conn}, that $M_1', M_2'$ are almost 3-connected: (for a proof see Lemma \ref{lem:almost3con3sum} below, where a stronger statement is proved). Moreover $M_1'$ and $M_2'$ are isomorphic to minors of $\si(M)$, hence of $M$. Let $A$ be obtained by adding to $A'$ all the elements of $E(M)$ that are parallel to some of $A'$, and define $B$ similarly. Clearly $(A,B)$ is a non-trivial 3-separation of $M$, and consider the corresponding $M_1, M_2$ satisfying $M=M_1\oplus_3 M_2$. For $i=1,2,$ $M_i$ is almost 3-connected, as by removing some parallel elements from $M_1$ we obtain $M_1'$; and $M_i$ is a minor of $M$, for the same reason.
\end{proof}

Now the proof of Theorem \ref{thm:seymourdecomposition3conn} follows, except for a last technicality which we now describe informally. Assume we first decompose our $M$ as $M_1\oplus M_2$, with $T$ being the common triangle, and that $M_2$ has a 3-separation $(A,B)$ ``crossing" $T$: say that two elements of $T$ are in $A$ and one is in $B$. Instead of using the separation $(A,B)$ to decompose $M_2$, we modify the separation in order to obtain another 3-separation that does not cross $T$ by moving the element in $T\cap B$ to $A$. We use the next fact that easily follows from the proof of Lemma~11.3.17 in~\cite{truemper1992matroid}. 

\begin{lemma}\label{prop:3separationtriangles}
Let $M$ be regular matroid with $R_{12}$ as a minor. Let $T_1,\dots,T_k$ be mutually disjoint triangles of $M$. Then $M$ has a non-trivial 3-separation $(A,B)$ such that each $T_i$ is contained in one of $A$ or $B$.
\end{lemma}

We are now ready to prove Theorem~\ref{thm:seymourdecomposition3conn}.
We remark that a similar result appears in \cite{dinitz2014matroid}, where result similar as Lemma~\ref{prop:3separationtriangles} is used. However, \cite{dinitz2014matroid} obtain a decomposition tree that may involve $1$- or $2$-sums.

\begin{proof}[Proof of Theorem \ref{thm:seymourdecomposition3conn}]
Let $M$ be 3-connected regular matroid that is not $R_{10}$. We obtain our final decomposition tree as the result of an iterative procedure, starting from a single node labeled $M$.

At each step, we require that the current tree $\mathcal{T}$ and its labels have the following properties: 

\begin{enumerate}[(i)]
\item each node $v$ of $\mathcal{T}$ is labeled with a regular matroid $M_v$ that is almost 3-connected,
\item each edge $vw \in E(\mathcal{T})$ has a corresponding $3$-sum $M_v \oplus_3 M_w$ (in particular $v, w$ are adjacent if and only if the $M_v, M_w$ share a triangle $T_{vw}$);
\item performing 3-sums over the all the edges, in arbitrary order, gives $M$ as a result.
\end{enumerate}

If all labels of the current tree $\mathcal{T}$ are graphic or cographic, then we are done. 

Assume otherwise, and let $v$ be a node of $\mathcal{T}$ whose label $M_v$ is not graphic or cographic.  Thanks to Lemmas \ref{prop:R10,R12} and \ref{lem:almost3con}, $M_v$ has an $R_{12}$ minor, hence a non-trivial $3$-separation. We will choose the separation in such a way no triangle of $M_v$ is ``crossed''. 

More precisely, let $u_1,\dots,u_k$ be neighbours of $v$ in $\mathcal{T}$, and $T_{1},\dots, T_{k}$ be the (disjoint) triangles of $M_v$ involved in the corresponding 3-sums, with $T_i=T_{vu_i}$ for each $i$. Thanks to Lemma \ref{prop:3separationtriangles}, $M_v$ has a non-trivial separation $(A,B)$ such that each $T_i$ is either in $A$ or $B$. 

Let $M_v = M_1\oplus_3 M_2$ be the 3-sum corresponding to separation $(A,B)$. In $\mathcal{T}$, we delete $v$ and add two adjacent nodes $v_1, v_2$, labeled $M_1, M_2$ respectively, and join $u_i$ to $v_1$ if $T_i\subseteq A$, and to $v_2$ otherwise. It is easy to check that, after these modifications, $\mathcal{T}$ still satisfies the required properties. 

Before concluding the proof, we remark that every time a matroid $M_v$ is split into two during the construction, the two corresponding matroids have strictly less elements than $M_v$ (and at least 7), which ensures that the procedure can only be repeated finitely many times. Hence we must at some point have that all labels are graphic or cographic, and the proof is complete.
\end{proof}

\section{Proving the assumptions for the cographic case}

In this section we argue that the assumption made at the beginning of Section \ref{sec:cographic} holds. Fix a 3-connected regular matroid $M$, different from $R_{10}$. In this section, we call any tree $\mathcal{T}$ satisfying the conditions of Theorem \ref{thm:seymourdecomposition3conn} (namely, each node $v \in V(\T)$ is labelled by a graphic or cographic matroid $M_v$, and performing 3-sums over the edges gives back $M$) a \emph{decomposition tree} of $M$. 

Consider a decomposition tree $\T$ of $M$, and a node $v \in V(\T)$ such that $M_v=M^*(G_v)$ is cographic (and not graphic, that is, $G_v$ is not planar). Let us call $v$ \emph{bad} if $G_v$ contains a cut $T_{uv}$ that is involved in some 3-sum and \emph{not} of the form $\delta(w)$ for some degree-$3$ node $w \in V(G_v)$. Also, we call $T_{uv}$ a \emph{bad} cut of $G_v$. Our goal it to show the following:

\begin{proposition}\label{prop:claws}
Every 3-connected regular matroid $M$ distinct from $R_{10}$ has a decomposition tree without bad nodes.
\end{proposition}

In order to prove Proposition~\ref{prop:claws}, we will start from any decomposition tree $\mathcal{T}$ and modify it until it has no bad nodes. At each step, we will maintain that $\T$ is a decomposition tree of $M$ and that each matroid $M_v$ for $v \in V(\T)$ is almost 3-connected. We can assume this last condition for the initial tree, see the proof of Theorem \ref{thm:seymourdecomposition3conn}. 

Now, we state some lemmas that will be useful for proving Proposition~\ref{prop:claws}. We recall that all matroids considered in this paper have no loops, as this is implicitly used in the proofs below.
 
 \begin{lemma}\label{lem:almost3coniff}
  A matroid $M$ is almost 3-connected if and only if it is connected, and for any 2-separation $(A,B)$ of $M$ we have that $\rk(A)=1$ or $\rk(B)=1$.
 \end{lemma}
 
 \begin{proof}
We start by proving the ``only if'' direction. 

First, assume that $M$ has a 1-separation. Then it is easy to see that $M$ has a 1-separation $(A,B)$ such that no element in $A$ is parallel to an element in $B$. But then by deleting elements we obtain that $\si(M)$ has a 1-separation, a contradiction. Hence, $M$ has no $1$-separation.

Now assume that $M$ has a 2-separation $(A,B)$. We may assume that $\rk(B)>1$. Hence, $B$ contains two elements that are not parallel. Since parallelism is symmetric and transitive, this implies that there are elements $a\in A$ and $b\in B$ that are not parallel. Then $\si(M)$ can be obtained by deleting all elements parallel to $a$, all elements parallel to $b$, and possibly others. Consider the partition $(A',B')$ of $E(\si(M))$ obtained from $(A,B)$ by deleting parallel elements in this way. Obviously, $a \in A'\subseteq A$, $b \in B'\subseteq B$, and $\rk(A') \leq \rk(A)$ and $\rk(B') \leq \rk(B)$. Since $(A',B')$ cannot be a $1$-separation of $\si(M)$, we have in fact $\rk(A') = \rk(A)$ and $\rk(B') = \rk(B) > 1$. Since $(A',B')$ cannot be a 2-separation of $\si(M)$, it must be that $|A'|=1$. But then $\rk(A)=\rk(A')=1$, which is what we wanted to prove.

Finally, we prove the ``if'' direction. First, we notice that $\si(M)$ has no 1-separation. Indeed, otherwise adding back the parallel elements yields a 1-separation of $M$. Now, assume by contradiction that $\si(M)$ has a 2-separation $(A,B)$. Since $\si(M)$ has no parallel elements and $|A|,|B|\geq 2$, we have $\rk(A),\rk(B)\geq 2$, but then adding back the parallel elements we get a 2-separation of $M$ in which both parts have rank at least $2$, a contradiction.
\end{proof}
 
\begin{corollary}\label{cor:almost3congraph}
Let $M=M^*(G)$ be an almost 3-connected cographic matroid. Then $G$ is 2-connected. If $(A,B)$ is a 2-separation of $M$, then one of $A$ or $B$ consists of edges that form an induced path in $G$.
\end{corollary}
\begin{proof}
Both parts follow easily from Lemma \ref{lem:almost3coniff}. First, $M$ is connected, hence $G$ is 2-connected. Second, without loss of generality, we may assume that $\rk(A)=1$. Hence any two edges in $A$ form a minimal cut of $G$. This implies that the edges in $A$ must form an induced path in $G$.
\end{proof}
 
\begin{lemma}\label{lem:almost3con3sum}
If $M$ is a binary, almost 3-connected matroid and $M=M_1\oplus_3 M_2,$ then $M_1$ and $M_2$ are almost 3-connected.
\end{lemma}
\begin{proof}
We use Lemma \ref{lem:almost3coniff} to argue that $M_1$ (hence, by simmetry, $M_2$) is almost 3-connected.
First, assume by contradiction that $M_1$ has a 1-separation $(A,B)$. Since $M$ is connected, $|A|, |B|\geq 2$ (if for instance $A = \{a\}$ then $a$ is a loop or coloop of $M_1$, hence of $M$). Let $T := E(M_1)\cap E(M_2)$, as usual. By symmetry, we may assume that $|A \cap T| > |B \cap T|$. By moving at most one element from $B$ to $A$, we can assume that $T \subseteq A$. 

We claim that $M=M \restrict (A \cup E(M_2)) \setminus T \oplus_1 M\restrict B$, in contradiction with the connectedness of $M$. Indeed, using the definitions of 1-sum (or 1-separation) and 3-sum, we have that $C \subseteq E(M)$ is a cycle of $M$ if and only if $C=(C_A\Delta C_B)\Delta C_2$, where $C_2$ a cycle of $M_2$ and $C_A$, $C_B$ are (disjoint) cycles of $M_1 \restrict A, M_1 \restrict B$ respectively. This is equivalent to $C=(C_A\Delta C_2)\Delta C_B$, with $C_A \Delta C_2$ being a cycle of $M \restrict (A\cup E(M_2))\setminus T$ and $C_B$ being a cycle of $M \restrict B$, implying the claim. 
 
Now, assume that $M_1$ has an exact 2-separation $(A,B)$, with $\rk(A) > 1$ and $\rk(B) > 1$, and say that $|T\cap A|\geq 2$. We argue that we can assume, as before, that $T\subseteq A$. If not, then $T\cap B=\{b\}$ for some element $b$, and $\rk(A)=\rk(A+b)$. If $|B|\leq 2$ or $\rk(B-b)\leq 1,$ one can see that $(A+b, B-b)$ is a 1-separation of $M_1$, a contradiction as shown above. 
Hence $|B|\geq 3$, and $(A+b, B-b)$ is still a 2-separation with $\rk(B-b)>1$. Now, similarly as above, we have that a set $C$ is a cycle of $M$ if and only if $C=(C_A\Delta C_B)\Delta C_2$, with $C_A$, $C_B$ and $C_2$ as above. This is clearly equivalent to $C=(C_A\Delta C_2)\Delta C_B$. Therefore, $M$ is the 2-sum of $M\restrict(A\cup E(M_2) \setminus T$ and $M \restrict B$. Hence $M$ has a 2-separation $((A\cup E(M_2))\setminus T, B)$ with $\rk((A\cup E(M_2))\setminus T)>1$ and $\rk(B)>1$, a contradiction.
\end{proof}
 
In the proof of Proposition \ref{prop:claws}, we will consider a bad node $v$ and show that the cographic matroid $M_v$ can be decomposed as a 3-sum. Hence we will need an analogous argument as in Lemma \ref{prop:3separationtriangles}, which cannot be applied to cographic matroids. This time, instead of modifying a 3-separation so that it does not cross any triangle, we will modify the triangles by swapping (i.e. exchanging the name of) two parallel elements, as this does not affect any 3-sum in which $M_v$ is involved. 

\begin{lemma}\label{prop:parallel}
  Let $M=M_1\oplus_3 M_2$, $T=E(M_1)\cap E(M_2)$, and let $\alpha'$ be an element of $M_1$ that is parallel to $\alpha\in T$. Consider the matroid $M_2'$ obtained from $M_2$ by renaming $\alpha$ as $\alpha'$. Then $M$ is isomorphic to $M_1\oplus_3 M_2'$.
\end{lemma}
\begin{proof}
In particular, we claim that $M_1\oplus_3 M_2'$ is obtained from $M$ by renaming $\alpha'$ as $\alpha$. This is an immediate consequence of the fact that, since $\alpha,\alpha'$ are parallel in $M_1,$ a set $C$ is a cycle of $M_1$ if and only if $C\Delta \{\alpha,\alpha'\}$ is. 
\end{proof}

Let $G$ be a 2-connected graph.
Consider two disjoint bad cuts of size 3 $T_i$, $T_j$, and let $(E^i_1,E^i_2,T_i)$ be the partition of $E(G)$ induced by $T_i$, and similarly for $E^j_1,E^j_2$. Since the cuts are bad, we have that $E^i_1$, $E^i_2$, $E^j_1$ and $E^j_2$ are all non-empty. We say that $T_j$ \emph{crosses} $T_i$ if $T_j$ has non-empty intersection with both $E^i_1$ and $E^i_2$. It is easy to check, using the fact that $G$ is 2-connected, that $T_j$ crosses $T_i$ if and only if $T_i$ crosses $T_j$. Hence we just say that $T_i$ and $T_j$ cross. We now show, using standard arguments, that there is a simple procedure to ``uncross'' all the bad cuts of $G$. In the following, a 3-cut of $G$ denotes a cut of size 3 of $G$.

\begin{lemma}\label{lem:uncross}
  Let $G$ be a 2-connected graph, and let $T_1,\dots, T_k$ be disjoint $3$-cuts of $G$.
  \begin{enumerate}[(i)]
      \item Assume that $T_i$ and $T_j$ cross. Then there are elements $e\in T_i$, $f\in T_j$ such that $\{e,f\}$ is a minimal cut of $G$ (that is, $e$ and $f$ are parallel in $M^*(G)$), and $T_i-e+f$, $T_j-f+e$ are disjoint 3-cuts that do not cross.
      \item By exchanging parallel elements between pairs of crossing cuts we end up with $k$ disjoint $3$-cuts that mutually do not cross.
  \end{enumerate}
\end{lemma}
\begin{proof}
(i) Denote by $(V_1^i,V_2^i)$ the partition of $V(G)$ induced by $T_i$, and define $(V_1^j,V_2^j)$ analogously, so that the subgraph of $G$ induced by $V^i_1$ has edge set $E_1^i$, and similarly for $V^i_2, E^i_2$, $V^j_1, E^j_1,$ and $V^j_2,E^j_2$. We claim that each of $V_1^i, V_2^i$ has non-empty intersection with each of $V_1^j, V_2^j$: indeed, if, say, $V_1^i\cap V_1^j=\emptyset$, i.e.\ $V^i_1\subseteq V_2^j$, then the endpoints of any edge of $T_j$ cannot be both in $V^i_1$ implying that $T_j\cap E^i_1=\emptyset$. 

We can assume without loss of generality that two edges of $T_j$ have endpoints in $V_1^i$, and one, which we denote by $f$, has endpoints in $V_2^i$ (see Figure \ref{fig:cross}). In the same way, we can assume that one edge $e$ of $T_i$ has its endpoints in $V_1^j$, and two in $V_2^j$. But then, $e,f$ form a cut of $G$, and $T_i-e+f=\delta(V_2^i\cap V_2^j)$ and $T_j-f+e=\delta(V_1^i\cap V_1^j)$ are two disjoint 3-cuts that do not cross.\medskip

(ii) We show that, whenever we swap two elements between two cuts that cross, the total number of 3-cuts that cross strictly decreases, hence by repeating this we must uncross all the cuts. To this end, let $T_i, T_j$ be crossing just as above, and denote by $T_i'=\delta(V_2^i\cap V_2^j), T_j'=\delta(V_1^i\cap V_1^j)$ the ``new" 3-cuts after the swap. Assume that there is another 3-cut $T_h$ that does not cross $T_i$, but crosses $T_i'$. We claim that then $T_h$ does not cross $T_j'$, and crosses $T_j$, hence by replacing $T_i, T_j$ with $T_i',T_j'$ in our family of 3-cuts the number of cuts that cross strictly decreases. Indeed, since $T_h$ crosses $T_i'$, it has non-empty intersection with both $E_2^i\cap E_2^j$ and $E^i_1(G)\cup E^j_1\setminus (T_i\cup T_j)$; but, since $T_h$ does not cross $T_i$, this implies that it is contained in $E^i_2,$ and has non-empty intersection with $E_2^i\cap E_1^j$. But then $T_h$ crosses $T_j$, and it does not cross $T_j'$, as claimed.
\end{proof}

\begin{figure}[h]
    \centering
    
    \begin{tikzpicture}
    \tikzstyle{graph}=[draw, rounded corners=5, minimum size=2.5cm, anchor=center] 
    \tikzstyle{vertex} = [draw, circle, text centered, anchor=center, fill=white]
 
    \node (V) [graph,minimum size=5cm] {};
    \draw (V.west) -- (V.east);
   \draw (V.north) -- (V.south);
    
    \node (v1i) [xshift=-30, yshift=9] at (V.north) {$V^i_1$};
     \node (v2i) [xshift=30, yshift=9] at (V.north) {$V^i_2$};
      \node (v1j) [xshift=-9, yshift=30] at (V.west) {$V_2^j$};
     \node (v2j) [xshift=-9, yshift=-30] at (V.west) {$V_1^j$};
     
     \node (Ti) [color=blue, xshift=0, yshift=-9] at (V.south) {$T_i$};
     \node (Tj) [color=red, xshift=9, yshift=0] at (V.east) {$T_j$};
     
     \node (a) [vertex, xshift=-45,yshift=15] at (V.center) {};
     \node (b) [vertex, xshift=-25,yshift=15] at (V.center) {};
     \node (c) [vertex, xshift=-15,yshift=58] at (V.center) {};
     \node (d) [vertex, xshift=-15,yshift=38] at (V.center) {};
     \node (a') [vertex, xshift=-45,yshift=-15] at (V.center) {};
     \node (b') [vertex, xshift=-25,yshift=-15] at (V.center) {};
     \node (c') [vertex, xshift=15,yshift=58] at (V.center) {};
     \node (d') [vertex, xshift=15,yshift=38] at (V.center) {};
     
    \node (e) [vertex, xshift=-15,yshift=-45] at (V.center) {};
     \node (e') [vertex, xshift=15,yshift=-45] at (V.center) {};
     
    \node (f) [vertex, xshift=45,yshift=15] at (V.center) {};
     \node (f') [vertex, xshift=45,yshift=-15] at (V.center) {};
     
     \node [red, xshift=52,yshift=6] at (V.center) {\large $f$};
      \node [blue, xshift=-5,yshift=-38] at (V.center) {\large $e$};
      
     \draw [color=red,ultra thick] (a)--(a');
      \draw [color=red,ultra thick] (b)--(b');
       \draw [color=blue,ultra thick] (c)--(c');
        \draw [color=blue,ultra thick] (d)--(d');
        \draw [color=blue,ultra thick] (e)--(e');
        \draw [color=red,ultra thick] (f)--(f');

    \end{tikzpicture}
    \caption{Crossing cuts.}
    \label{fig:cross}
\end{figure}
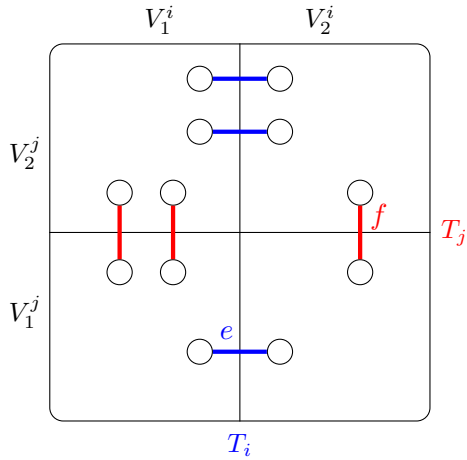

We are now ready to prove Proposition \ref{prop:claws}.

\begin{figure}
    \centering
\begin{tikzpicture}[remember picture]
\tikzstyle{vertex} = [draw, circle,
 text centered, anchor=center, fill=white]

  \tikzstyle{graph} = [ draw, circle, text centered, minimum size=2.3cm]

\node (g1) [graph,  fill=blue!10] {$G_1$};

\node (g2) [graph,  fill=blue!10, xshift=3.25cm,yshift=0cm] at (g1.center) {$G_2$};

\node (b1) [vertex,xshift=1.1cm,yshift=0cm] at (g1.center) {} ;

\node (a1) [vertex,xshift=0cm,yshift=0.45cm] at (b1.north) {} ;

\node (c1) [vertex,xshift=0cm,yshift=-0.45cm] at (b1.south) {} ;

\node (b2) [vertex,xshift=-1.1cm,yshift=0cm] at (g2.center) {} ;

\node (a2) [vertex,xshift=0cm,yshift=0.45cm] at (b2.north) {} ;

\node (c2) [vertex,xshift=0cm,yshift=-0.45cm] at (b2.south) {} ;

\draw [-] (a1) -- (a2);
\draw [-] (b1) -- (b2);
\draw [-] (c1) -- (c2);

\node [xshift=10,yshift=4] at (a1.north east) {$\alpha$};
\node [xshift=10,yshift=4] at (b1.north east) {$\beta$};
\node [xshift=10,yshift=4] at (c1.north east) {$\gamma$};

\node (eq) [xshift=1.55cm, yshift=-0.7cm] at (g1.south) {=};

\node (g1') [graph,  fill=blue!10,xshift=-2.4cm, yshift=-2cm] at (g1.south) {$G_1$};

\node (b1') [vertex,xshift=1.1cm,yshift=0cm] at (g1'.center) {} ;

\node (a1') [vertex,xshift=0cm,yshift=0.45cm] at (b1'.north) {} ;

\node (c1') [vertex,xshift=0cm,yshift=-0.45cm] at (b1'.south) {} ;

\node (v) [vertex,xshift=1cm,yshift=0cm] at (b1'.east) {} ;

\draw [-] (v) -- (a1');
\draw [-] (v) -- (b1');
\draw [-] (v) -- (c1');

\node [xshift=0.08cm, yshift=-1.1cm] at (eq.south) {$\oplus_3$};

\node (g2') [graph,  fill=blue!10,xshift=8.6cm, yshift=0cm] at (g1'.center) {$G_2$};

\node (b1'') [vertex,xshift=-1.1cm,yshift=0cm] at (g2'.center) {} ;

\node (a1'') [vertex,xshift=0cm,yshift=0.45cm] at (b1''.north) {} ;

\node (c1'') [vertex,xshift=0cm,yshift=-0.45cm] at (b1''.south) {} ;

\node (a2'') [vertex,xshift=-1cm,yshift=0cm] at (a1''.center) {} ;

\node (b2'') [vertex,xshift=-1cm,yshift=0cm] at (b1''.center) {} ;

\node (c2'') [vertex,xshift=-1cm,yshift=0cm] at (c1''.center) {} ;

\node (v') [vertex,xshift=-1cm,yshift=0cm] at (b2''.west) {} ;

\draw [-] (v') -- (a2'');
\draw [-] (v') -- (b2'');
\draw [-] (v') -- (c2'');
\draw [-] (a1'') -- (a2'');
\draw [-] (b1'') -- (b2'');
\draw [-] (c1'') -- (c2'');

\node [xshift=10,yshift=4] at (a2''.north east) {$\alpha$};
\node [xshift=10,yshift=4] at (b2''.north east) {$\beta$};
\node [xshift=10,yshift=4] at (c2''.north east) {$\gamma$};

\node [yshift=-2.8cm] at (eq.south) {$=$};

\end{tikzpicture}

\vspace{0.5cm}

\hspace{-0.4cm}\begin{tikzpicture}[remember picture]
\tikzstyle{vertex} = [draw, circle, 
 text centered,anchor=center, fill=white]

  \tikzstyle{graph} = [ draw, circle, text centered, minimum size=2.3cm]

\node (g1') [graph,  fill=blue!10] {$G_1$};

\node (b1') [vertex,xshift=1.1cm,yshift=0cm] at (g1'.center) {} ;

\node (a1') [vertex,xshift=0cm,yshift=0.45cm] at (b1'.north) {} ;

\node (c1') [vertex,xshift=0cm,yshift=-0.45cm] at (b1'.south) {} ;

\node (v) [vertex,xshift=1cm,yshift=0cm] at (b1'.east) {} ;

\draw [-] (v) -- (a1');
\draw [-] (v) -- (b1');
\draw [-] (v) -- (c1');

\node [xshift=1cm] at (v.east) {$\oplus_3$};

\node (w) [vertex,xshift=1.8cm,yshift=0cm] at (v.east) {} ;

\node (b1'') [vertex,xshift=1cm,yshift=0cm] at (w.center) {} ;

\node (a1'') [vertex,xshift=0cm,yshift=0.75cm] at (b1''.north) {} ;

\node (c1'') [vertex,xshift=0cm,yshift=-0.75cm] at (b1''.south) {} ;

\node (a2) [vertex,xshift=1cm,yshift=0cm] at (a1''.east) {} ;

\node (b2) [vertex,xshift=1cm,yshift=0cm] at (b1''.east) {} ;

\node (c2) [vertex,xshift=1cm,yshift=0cm] at (c1''.east) {} ;

\node (w') [vertex,xshift=3cm,yshift=0cm] at (w.east) {} ;

\draw [-] (a1'') -- (w);
\draw [-] (b1'') -- (w);
\draw [-] (c1'') -- (w);

\draw [-] (a1'') -- (a2);
\draw [-] (b1'') -- (b2);
\draw [-] (c1'') -- (c2);

\draw [-] (w') -- (a2);
\draw [-] (w') -- (b2);
\draw [-] (w') -- (c2);

\node [xshift=10,yshift=4] at (a1''.north east) {$\alpha$};
\node [xshift=10,yshift=4] at (b1''.north east) {$\beta$};
\node [xshift=10,yshift=4] at (c1''.north east) {$\gamma$};

\node [xshift=1cm, yshift=0cm] at (w'.east) {$\oplus_3$};

\node (g2') [graph,  fill=blue!10,xshift=11.8cm, yshift=0cm] at (g1'.center) {$G_2$};

\node (b1'') [vertex,xshift=-1.1cm,yshift=0cm] at (g2'.center) {} ;

\node (a1'') [vertex,xshift=0cm,yshift=0.45cm] at (b1''.north) {} ;

\node (c1'') [vertex,xshift=0cm,yshift=-0.45cm] at (b1''.south) {} ;

\node (v') [vertex,xshift=-1cm,yshift=0cm] at (b1''.west) {} ;

\draw [-] (v') -- (a1'');
\draw [-] (v') -- (b1'');
\draw [-] (v') -- (c1'');
\end{tikzpicture}
\caption{An example of the 3-sum decomposition that we do in the proof of Proposition \ref{prop:claws}. In at most two decompositions, the bad cut $\{\alpha,\beta,\gamma\}$ is moved to a cographic matroid that is also graphic. }
\label{fig:cographic3}
\end{figure}
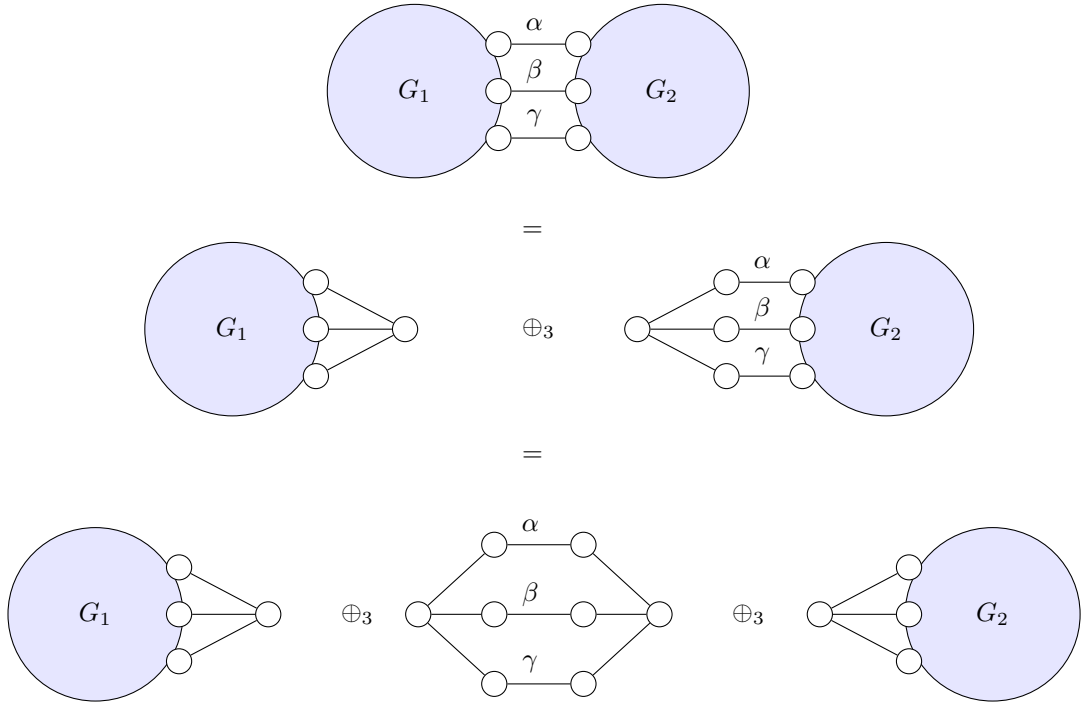

\begin{proof}[Proof of Proposition \ref{prop:claws}]
We start from any decomposition tree $\T$ of $M$ (see Theorem~\ref{thm:seymourdecomposition3conn}). If $\T$ has no bad node, then we are done. We may assume that $\T$ has a bad node. Let $M^*(G)$ be the corresponding cographic matroid. Then $M^*(G)$ is almost 3-connected and not graphic. By Corollary \ref{cor:almost3congraph}, $G$ is 2-connected (and not planar). Graph $G$ has disjoint 3-cuts $T_{1}, \dots, T_{k}$, each involved in a 3-sum, some of which are bad. 

Let $T$ be one of the bad cuts. Denote by $(V_1, V_2)$ the partition of $V(G)$ given by $T = \{\alpha,\beta,\gamma\}$, and by $U_1,U_2$ the sets of endpoints of $\alpha,\beta,\gamma$ that are in $V_1,V_2$ respectively.  Denote by $E_1, E_2$ the edge sets of $G[V_1], G[V_2]$ respectively. By symmetry we may assume that $|U_1|\geq |U_2|$. Notice that, since $T$ is bad and $G$ is 2-connected, none of $E_1, E_2$ can be empty, and we must have $|U_2|\geq 2$.

First, assume $|U_2|=2$. It is easy to check that either $|E_2|=1$ or $(E_1\cup T, E_2)$ is a 2-separation of $M^*(G)$. In the latter case, thanks to Corollary \ref{cor:almost3congraph}, we have that $E_2$ consists of an induced path of $G$. This also trivially holds in the former case. Let $U_2=\{u,v\}$, so that $u$ is incident to one edge of $T$, say $\alpha$, and $v$ to the other two. Denote by $\alpha'$ the edge of $E_2$ incident to $v$. It is immediate to see that $\alpha, \alpha'$ are parallel in $M^*(G)$. Then, we can swap $\alpha$ and $\alpha'$ while leaving $M^*(G)$ and in general $M$ unchanged, by Lemma \ref{prop:parallel}. Doing so, we strictly decrease the number of bad cuts of $T$, as now $T=\delta(v)$ is a good cut, and no other good cut of $G$ is modified by the swap as it cannot have an edge in $E_2$. 

Hence, by iterating the above procedure we can assume that, if $G$ still has a bad cut $T$, it consists of three pairwise non-incident edges. We argue that in this case $M^*(G)$ has a non-trivial 3-separation. Assume without loss of generality that $|E_1|\geq |E_2|$. We show that $|E_1|\geq 4$. If, by contradiction, $|E_1|\leq 3$, we have that $G$ has at most 9 edges, but since $G$ is not planar, $G$ must be $K_{3,3}$. But $K_{3,3}$ cannot have a bad 3-cut, a contradiction. Hence, $|E_1| \geq 4$.

Now it is easy to check that $(E_1,E_2\cup T)$ is a non-trivial 3-separation of $M^*(G)$. Similarly as in the proof of Theorem \ref{thm:seymourdecomposition3conn}, we can only decompose $M^*(G)$ along such separation if there is no other 3-cut $T'$ that crosses the separation (note that such $T'$ would need to be a bad cut).  But thanks to Lemma \ref{lem:uncross} we can swap parallel elements between the bad cuts of $G$ until no two of them cross, and then we can decompose over any non-trivial 3-separation induced by a bad cut.
Hence we can decompose $M^*(G)$ as a 3-sum of two smaller matroids $M_1, M_2$ and modify $\mathcal{T}$ as in the proof of Theorem \ref{thm:seymourdecomposition3conn}. Thanks to Lemma \ref{lem:almost3con3sum}, $M_1, M_2$ are still almost 3-connected, and it is easy to convince oneself (see Figure \ref{fig:cographic3}) that $M_1, M_2$ are still cographic, hence we can iterate the argument until $\T$ has no more bad nodes.
\end{proof}

\end{document}